\documentclass[11pt]{article}
\usepackage[english]{babel}
\usepackage{amssymb,amsmath,amsthm}
\newtheorem{theorem}{Theorem}[section]
\newtheorem{lemma}[theorem]{Lemma}
\newtheorem{proposition}[theorem]{Proposition}
\newtheorem{corollary}[theorem]{Corollary}
\newtheorem{question}[theorem]{Question}

{\theoremstyle{definition}\newtheorem{definition}[theorem]{Definition}}
{\theoremstyle{definition}\newtheorem{remark}[theorem]{Remark}}

\numberwithin{equation}{section}

\def\C{{\mathbb C}}
\def\K{{\mathbb K}}
\def\N{{\mathbb N}}
\def\Z{{\mathbb Z}}
\def\R{{\mathbb R}}

\def\T{{\mathbb T}}

\def\hh{{\mathcal H}}

\def\qq{{\mathbb Q}}
\def\pd{p_{{}_{\scriptstyle D}}}
\def\epsilon{\varepsilon}
\def\phi{\varphi}
\def\leq{\leqslant}
\def\geq{\geqslant}
\def\ker{{\tt ker}\,}
\def\spann{\hbox{\tt span}\,}
\def\supp{\hbox{\tt supp}\,}
\def\spannn{\hbox{$\overline{\hbox{\tt span}}$}\,}

\title{Orbital and strongly orbital spaces}

\author{Stanislav Shkarin}

\date{}

\begin{document}

\maketitle

\begin{abstract}We say that a (countably dimensional) topological
vector space $X$ is orbital if there is $T\in L(X)$ and a vector
$x\in X$ such that $X$ is the linear span of the orbit
$\{T^nx:n=0,1,\dots\}$. We say that $X$ is strongly orbital if,
additionally, $x$ can be chosen to be a hypercyclic vector for $T$. Of course, 
$X$ can be orbital only if the algebraic dimension of $X$ is finite or 
infinite countable. We characterize orbital and strongly orbital metrizable locally
convex spaces. We also show that every countably dimensional
metrizable locally convex space $X$ does not have the invariant
subset property. That is, there is $T\in L(X)$ such that every
non-zero $x\in X$ is a hypercyclic vector for $T$. Finally, assuming
the Continuum Hypothesis, we construct a complete strongly orbital
locally convex space.

As a byproduct of our constructions, we determine the number of
isomorphism classes in the set of dense countably dimensional
subspaces of any given separable infinite dimensional Fr\'echet
space $X$. For instance, in $X=\ell_2\times \omega$, there are
exactly 3 pairwise non-isomorphic (as topological vector spaces)
dense countably dimensional subspaces.
\end{abstract}

\normalsize

\section{Introduction \label{s1}}\rm

All vector spaces in this article are over the field $\K$  being
either the field $\C$ of complex numbers or the field $\R$ of real
numbers. As usual, $\T=\{z\in\C:|z|=1\}$, $\N$ is the set of
positive integers and $\Z_+=\N\cup\{0\}$. Throughout the article,
all topological spaces {\it are assumed to be Hausdorff}. Recall
that a {\it Fr\'echet space} is a complete metrizable locally convex
space. For a topological vector space $X$, $L(X)$ is the algebra of
continuous linear operators on $X$ and $X'$ is the space of
continuous linear functionals on $X$. Symbol $GL(X)$ stands for the
group of invertible $T\in L(X)$ for which $T^{-1}\in L(X)$. For
$T\in L(X)$, the dual operator $T':X'\to X'$ is defined as usual:
$T'f=f\circ T$.

By saying {\it countable}, we always mean infinite countable. Recall
that the topology $\tau$ of a topological vector space $X$ is called
{\it weak} if $\tau$ is the weakest topology making each $f\in Y$
continuous for some linear space $Y$ of linear functionals on $X$
separating points of $X$. It is well-known and easy to see that a
topology of a metrizable infinite dimensional topological vector
space $X$ is weak if and only if $X$ is isomorphic to a dense linear
subspace of $\omega=\K^\N$.

Recall that a topological vector space $X$ has the {\it invariant
subspace property} if every $T\in L(X)$ has a non-trivial
(=different from $\{0\}$ and $X$) closed invariant subspace.
Similarly, a topological vector space $X$ has the {\it invariant
subset property} if every $T\in L(X)$ has a non-trivial (=different
from $\varnothing$, $\{0\}$ and $X$) closed invariant subset. The
problem whether $\ell_2$ has the invariant subspace property is
known as the invariant subspace problem and remains perhaps the most
famous open problem in operator theory. The problem whether $\ell_2$
has the invariant subset property is also open. It is worth noting
that Read \cite{read1} and Enflo \cite{enf} (see also
\cite{isp-book}) showed independently that there are separable
infinite dimensional Banach spaces, which do not have the invariant
subspace property. In fact, Read \cite{read2,read3} demonstrated
that $\ell_1$ does not have the invariant subset property. It is
worth noting that Atzmon \cite{atz1,atz2} constructed an infinite
dimensional nuclear Fr\'echet space without the invariant subspace
property. All existing constructions of operators on Banach spaces
with no invariant subspaces are rather sophisticated. On the other
hand, examples of separable non-complete normed spaces with or
without the invariant subspace or subset property are relatively
easy to construct.

Recall that $x\in X$ is called a {\it hypercyclic vector} for $T\in
L(X)$ if the orbit
$$
O(T,x)=\{T^nx:n\in\Z_+\}
$$
is dense in $X$. Similarly, $x$ is called {\it cyclic} for $T$ if
the linear span of $O(T,x)$ is dense in $X$. Clearly, $T$ has no
non-trivial invariant subsets (respectively, subspaces) precisely
when every non-zero vector is hypercyclic (respectively, cyclic) for
$T$. If the space in question is countably dimensional, there is an
unusual way to approach the invariant subspace/subset property.

\begin{definition}\label{d1} We say that a
topological vector space $X$ is {\it orbital} if there exist $T\in
L(X)$ and $x\in X$ such that $X=\spann(O(T,x))$. We say that $X$ is
{\it strongly orbital} if there exist $T\in L(X)$ and $x\in X$ such
that $x$ is a hypercyclic vector for $T$ and $X=\spann(O(T,x))$.
\end{definition}

Note that an orbital space always has either finite or countable
dimension. Since finite dimensional topological vector spaces
support no hypercyclic operators, every strongly orbital space is
countably dimensional. We start with the following easy observation.

\begin{lemma}\label{orbi} Each strongly orbital topological vector
space does not have the invariant subset property.
\end{lemma}

\begin{proof} Let $X$ be a topological vector
space and $T\in L(X)$ and $x\in X$ be such that $x$ is a hypercyclic
vector for $T$ and $X=\spann(O(T,x))$. Due to Wengenroth \cite{ww},
$p(T)(X)$ is dense in $X$ for every non-zero polynomial $p$. It
suffices to show that $T$ has no non-trivial invariant subsets. Take
$y\in X\setminus\{0\}$. Since $X=\spann(O(T,x))$, there is a
non-zero polynomial $p$ such that $y=p(T)x$. Hence
$$
O(T,y)=O(T,p(T)y)=p(T)(O(T,x))
$$
and therefore $O(T,y)$  is dense in $X$ since each $p(T)(X)$ is
dense in $X$ and $O(T,x)$ is dense in $X$. Thus every non-zero
vector is hypercyclic for $T$, which means that $T$ has no
non-trivial invariant subsets.
\end{proof}

The following two propositions are known facts and are basically a
compilation of certain results from \cite{fre} and \cite{gri2}. We
present their proofs for the sake of convenience.

\begin{proposition}\label{prop1} Every normed space $X$ of countable
algebraic dimension is strongly orbital and therefore does not have
the invariant subset property.
\end{proposition}

\begin{proof} Let $B$ be the completion of $X$. Then $B$ is a separable infinite
dimensional Banach space. According to Ansari \cite{ansa1} and
Bernal--Gonz\'ales \cite{bernal}, there is a hypercyclic $T\in
L(B)$. That is, there is $x\in B$ such that $\{T^nx:n\in\Z_+\}$ is
dense in $B$. Let $Z$ be the linear span of $T^nx$ for $n\in\Z_+$.
Since both $Z$ and $X$ are dense countably dimensional linear
subspaces of $B$, according to Grivaux \cite{gri2}, there is $S\in
GL(B)$ such that $S(Z)=X$. Since $Z$ is invariant for $T$, $X$ is
invariant for $STS^{-1}$. Hence, the restriction
$A=STS^{-1}\bigr|_{X}$ belongs to $L(X)$, $O(A,Sx)=S(O(T,x))$ is
dense and $X=\spann(O(A,Sx))$. Thus $X$ is strongly orbital. By
Lemma~\ref{orbi}, $X$ does not have the invariant subset property.
\end{proof}

\begin{proposition}\label{prop1a} In every separable infinite dimensional
Fr\'echet space $X$, there is a dense linear subspace $E$ such that
$E$ has the invariant subspace property.
\end{proposition}

\begin{proof}
There is a dense linear subspace $E$ of $X$ ($E$ can even be chosen
to be a hyperplane \cite{bonet}) such that every $T\in L(E)$ has the
shape $\lambda I+S$ with $\lambda\in \K$ and $\dim S(X)<\infty$.
Trivially, such a $T$ has a one-dimensional invariant subspace.
\end{proof}

\subsection{Results}

We partially extend Proposition~\ref{prop1} from the class of normed
spaces to the class of metrizable locally convex (topological
vector) spaces.

\begin{theorem}\label{th1} Every metrizable locally convex space $E$ of
countable algebraic dimension does not have the invariant subset
property.
\end{theorem}

It turns out that not every metrizable locally convex space $X$ of
countable algebraic dimension is strongly orbital or even orbital.
In order to formulate the result neatly, we split the class ${\cal
M}$ of infinite dimensional metrizable locally convex spaces into
$4$ subclasses. Given $X\in{\cal M}$ and an increasing sequence
$\{p_n\}_{n\in\N}$ of seminorms on $X$ defining the topology of $X$,
we have the following $4$ possibilities:
\begin{itemize}\itemsep-3pt
\item[($A_0$)]$X/\ker p_1$ is finite dimensional and $\ker p_n/\ker
p_{n+1}$ is finite dimensional for each $n\in\N$;
\item[($A_1$)]there is $n\in\N$ such that $\ker p_n=\{0\}$;
\item[($A_2$)]$X/\ker p_1$ is infinite dimensional, $\ker p_n/\ker
p_{n+1}$ is finite dimensional for each $n\in\N$ and is non-zero for
infinitely many $n\in\N$;
\item[($A_3$)]$\ker p_n/\ker p_{n+1}$ is infinite dimensional for
infinitely many $n\in\N$.
\end{itemize}
It is an easy exercise (left to the reader) to show that exactly one
of the conditions ($A_0$--$A_3$) is satisfied for each $X$ and
$\{p_n\}$ and that each of the conditions ($A_0$--$A_3$) does not
depend on the choice of an increasing sequence $\{p_n\}_{n\in\N}$ of
seminorms defining the topology of $X\in{\cal M}$. Thus, conditions
($A_0$--$A_3$) split the class ${\cal M}$ into four disjoint
subclasses ${\cal M}_j$ with $0\leq j\leq 3$, where $X\in{\cal M}_j$
if for some (=for each) increasing sequence $\{p_n\}_{n\in\N}$ of
seminorms on $X$ defining the topology of $X$, the condition $(A_j)$
is satisfied.

\begin{remark}\label{MJ}
It is easy to show that the following alternative way of defining
the classes ${\cal M}_j$ holds. Namely, let $X\in{\cal M}$. Then
\begin{itemize}\itemsep-3pt
\item $X\in {\cal M}_0$ if and only if the topology of $X$ is weak;
\item $X\in{\cal M}_1$ if and only if $X$ possesses a continuous
norm;
\item $X\in {\cal M}_2$ if and only if there is a closed linear
subspace $Y$ of $X$ such that the topology of $Y$ is weak, $X/Y$
possesses a continuous norm and both $Y$ and $X/Y$ are infinite
dimensional;
\item $X\in {\cal M}_3$ if and only if the topology of $X$ can be
defined by an increasing sequence $\{p_n\}_{n\in\N}$ of seminorms
such that $\ker p_n/\ker p_{n+1}$ is infinite dimensional for every
$n\in\N$.
\end{itemize}
\end{remark}

\begin{theorem}\label{th1a} Let $E$ be a metrizable locally convex
space of countable algebraic dimension. Then
$$
\text{$E$ is orbital\ }\iff\text{\ $E$ is strongly orbital\ }\iff \
E\notin {\cal M}_2.
$$
\end{theorem}

The above theorem characterizes (strongly) orbital metrizable
countably dimensional locally convex topological vector spaces.
Proposition~\ref{prop1a} indicates that completeness is an essential
difficulty in constructing operators with no non-trivial invariant
subspaces. Countable algebraic dimension is often perceived as
almost incompatible with completeness. Basically, there is only one
complete topological vector space of countable dimension, most
analysts are aware of. Namely, the locally convex direct sum $\phi$
(see \cite{rob}) of countably many copies of the one-dimensional
space $\K$ has countable dimension and is complete. In other words,
$\phi$ is a vector space of countable dimension endowed with the
topology defined by the family of {\it all$\,$} seminorms. It is
easy to see that every linear subspace of $\phi$ is closed, which
leads to the following observation. It features (many times) in the
literature, see, for instance, \cite{koe}. We present the proof for 
the sake of completeness.

\begin{proposition}\label{prop2} The space $\phi$ does not support a
cyclic operator with dense range $($and therefore has the invariant
subspace property$).$
\end{proposition}

\begin{proof} Assume that $T\in L(\phi)$ has dense range and $x\in\phi$.
We have to show that $x$ is not a cyclic vector for $T$. If $T^nx$
for $n\in\Z_+$ are not linearly independent, then $\spann(O(T,x))$
is finite dimensional and therefore $x$ is not a cyclic vector for
$T$. If $T^nx$ for $n\in\Z_+$ are linearly independent, then
$E=\spann(O(T,x))\neq T(E)=\spann(O(T,Tx))$. Indeed, in this case
$T(E)$ is a hyperplane in $E$. Hence $T(E)\neq\phi$. Since every
linear subspace of $\phi$ is closed, $T(E)$ is not dense in $\phi$.
Since $T$ has dense range, $E$ is not dense in $\phi$. Since
$E=\spann(O(T,x))$, $x$ is not a cyclic vector for $T$.
\end{proof}

\begin{remark}\label{phirem} The space $\phi$ is orbital.
If one is after an analogue of Lemma~\ref{orbi} for the invariant
subspace property, a modification of orbitality is needed. Indeed,
one can show that if $X$ is a countably dimensional locally convex
space $x\in X$, $T\in L(X)$, $X=\spann(O(T,x))$ and the point
spectrum of the dual operator $T'$ (we have to take the
complexification of $T'$ in the case $\K=\R$) is empty, then $X$
does not have the invariant subspace property. More specifically,
$T$ does not have non-trivial invariant subspaces. The proof goes
along the same lines as in Lemma~\ref{orbi}.
\end{remark}

In the proof of Proposition~\ref{prop1} we have used the fact (due
to Grivaux \cite{gri2}) that isomorphisms on a separable infinite
dimensional Banach space act transitively on the set of dense
countable linearly independent sets. This property for more general
spaces is studied in \cite{ss}. In this paper we are interested in
the following corollary of this result. Namely, it immediately
follows that $GL(X)$ acts transitively on the set ${\cal E}(X)$ of
dense countably dimensional subspaces of a separable Banach space
$X$. Equivalently, countably dimensional normed spaces are
isomorphic as topological vector spaces precisely when their
completions are isomorphic. We extend this result to Fr\'echet
spaces. The answer turns out to be rather unexpected. Again, we
split the class ${\cal F}$ of separable infinite dimensional
Fr\'echet spaces into four subclasses
$$
\text{${\cal F}_j={\cal F}\cap{\cal M}_j$ for $0\leq j\leq 3$.}
$$
Using the well-known facts that every infinite dimensional Fr\'echet
space, whose topology is weak, is isomorphic to $\omega$ and that
every isomorphic to $\omega$ subspace of a Fr\'echet space is
complemented, we see that Remark~\ref{MJ} implies the following
fact.

\begin{remark}\label{MJ1}
Let $X\in{\cal F}$. Then $X\in {\cal F}_0$ if and only if $X$ is
isomorphic to $\omega$ and $X\in {\cal F}_2$ if and only if $X$ is
isomorphic to $Y\times\omega$, where $Y$ is infinite dimensional and
possesses a continuous norm.
\end{remark}

The following result provides the number of isomorphism classes of
dense countably dimensional linear subspaces of a Fr\'echet space.
From now on, for a topological vector space $X$,
$$
\text{${\cal E}(X)$ the set of dense countably dimensional linear
subspaces of $X$.}
$$

\begin{theorem}\label{grgr11} Let $X\in{\cal F}$.
Then
\begin{itemize}\itemsep-2pt
\item[\rm(\ref{grgr11}.1)] $X\in{\cal F}_0\iff {\cal E}(X)\cap{\cal M}_0\neq\varnothing
\iff{\cal E}(X)\subseteq{\cal M}_0$. Furthermore, if $X\in{\cal
F}_0$, then every $F,G\in{\cal E}(X)$ are isomorphic $($as
topological vector spaces$);$
\item[\rm(\ref{grgr11}.2)]$X\notin{\cal F}_0\iff {\cal E}(X)\cap{\cal
M}_1\neq\varnothing$ and every $E,F\in {\cal E}(X)\cap{\cal M}_1$
are isomorphic.  Furthermore, $X\in{\cal F}_1\iff {\cal
E}(X)\subseteq{\cal M}_1;$
\item[\rm(\ref{grgr11}.3)]$X\in{\cal F}_2\cup{\cal F}_3\iff {\cal E}(X)\cap{\cal
M}_2\neq\varnothing$. If $X\in{\cal F}_3$, then every $E,F\in {\cal
E}(X)\cap{\cal M}_2$ are isomorphic. If $X\in{\cal F}_2$, then
${\cal E}(X)\cap{\cal M}_2$ splits into $($exactly$)$ two
isomorphism classes$;$
\item[\rm(\ref{grgr11}.4)]$X\in{\cal F}_3\iff {\cal E}(X)\cap{\cal
M}_3\neq\varnothing$. If $X\in{\cal F}_3$, then the number of
isomorphism classes in ${\cal E}(X)\cap{\cal M}_3$ is uncountable.
\end{itemize}
\end{theorem}

Note that isomorphism classes in ${\cal E}(X)$ are in the natural
one-to-one correspondence with the orbits in ${\cal E}(X)$ under the
natural action of $GL(X)$. For instance, (\ref{grgr11}.2) implies
that $GL(X)$ acts transitively on ${\cal E}(X)\cap{\cal M}_1$. The
above theorem will be proved in subsequent sections. Right now, we
derive the following corollary.

\begin{corollary}\label{grgr1} For $X\in{\cal F}$ let $\nu(X)$ be the number
of isomorphism classes of spaces in ${\cal E}(X)$. Then $\nu(X)=1$
if $X\in{\cal F}_0\cup{\cal F}_1$, $\nu(X)=3$ if $X\in{\cal F}_2$
and $\nu(X)$ is uncountable if $x\in{\cal F}_3$.
\end{corollary}

\begin{proof} We use Theorem~\ref{grgr11}. From (\ref{grgr11}.1) it follows that
$\nu(X)=1$ if $X\in{\cal F}_0$. According to (\ref{grgr11}.2),
$\nu(X)=1$ if $X\in{\cal F}_1$. By (\ref{grgr11}.2) and
(\ref{grgr11}.3), $\nu(X)=3$ if $X\in{\cal F}_2$. Indeed, in the
notation of Theorem~\ref{grgr11}, ${\cal E}=({\cal E}\cap {\cal
M}_1)\cup({\cal E}\cap {\cal M}_2)$ with ${\cal E}\cap {\cal M}_1$
providing one isomorphism class and ${\cal E}\cap {\cal M}_2$
consisting of two isomorphism classes. Finally, (\ref{grgr11}.4)
implies that $\nu(X)$ is uncountable if $X\in {\cal F}_3$.
\end{proof}

\begin{remark}\label{re1} Theorem~\ref{grgr11} completes nicely the
main result of \cite{ss}, which says that for $X\in{\cal F}$,
$GL(X)$ acts transitively on the set of dense countable linearly
independent subsets of $X$ if and only if $X\in{\cal F}_1$. Note
also that $\nu(X)$ (defined as in Corollary~\ref{grgr1}) does not
exceed $2^{\aleph_0}$ for every $X\in{\cal F}$. We conjecture that
$\nu(X)=2^{\aleph_0}$ for each $X\in{\cal F}_3$.
\end{remark}

Contrary to the common perception, there is an abundance of complete
topological vector spaces of countable dimension.

\begin{theorem}\label{main} There is a complete locally convex
topological vector space $X$ of countable algebraic dimension such
that $X$ does not have the invariant subspace property.
\end{theorem}

In other words, Theorem~\ref{main} provides a complete locally
convex topological vector space $X$ with $\dim X=\aleph_0$ and $T\in
L(X)$ such that $T$ does not have non-trivial closed invariant
subspaces. Assuming the Continuum Hypothesis, we can go even
further.

\begin{theorem}\label{main1} Under the assumption of the Continuum
Hypothesis, there is a complete locally convex topological vector
space $X$ of countable algebraic dimension such that $X$ is strongly
orbital and therefore does not have the invariant subset property.
\end{theorem}

\section{Auxiliary results}

Recall that a subset $D$ of a locally convex space $X$ is called a
{\it disk} if $D$ is bounded, convex and balanced (=is stable under
multiplication by any $\lambda\in\K$ with $|\lambda|\leq1$). The
symbol $X_D$ stands for the space $\spann(D)$ endowed with the norm
$\pd$ being the Minkowskii functional of the set $D$. Boundedness of
$D$ implies that the topology of $X_D$ is stronger than the one
inherited from $X$. A disk $D$ in $X$ is called a {\it Banach disk}
if the normed space $X_D$ is complete. It is well-known that a
sequentially complete disk is a Banach disk, see, for instance,
\cite{bonet}. In particular, a compact or sequentially compact disk
is a Banach disk.

The following result features as Lemma~4.1 in \cite{ss}.

\begin{lemma}\label{l11} Let $\{x_n\}_{n\in\Z_+}$ be a convergent to $0$ sequence
in a sequentially complete locally convex space $X$. Then the set
$D=\Bigl\{\sum\limits_{n=0}^\infty a_nx_n:a\in\ell_1,\ \|a\|_1\leq
1\Bigr\}$ is a Banach disk. Moreover, $E=\spann\{x_n:n\in \Z_+\}$ is
a dense linear subspace of the Banach space $X_D$.
\end{lemma}

We say that a continuous seminorm $p$ on a locally convex space $X$
is {\it non-trivial} if
$$
\ker p=\{x\in X:p(x)=0\}
$$
has infinite codimension in $X$. If $p$ is a continuous seminorm on
a locally convex space $X$, we say that $A\subset X$ is $p$-{\it
independent} if $p(z_1a_1+{\dots}+z_na_n)\neq 0$ for any $n\in\N$,
any pairwise different $a_1,\dots,a_n\in A$ and any non-zero
$z_1,\dots,z_n\in\K$. In other words, vectors $x+\ker p$ for $x\in
A$ are linearly independent in $X/\ker p$.

\begin{lemma}\label{l22} Let $X$ be a Fr\'echet space,
$\{X_j\}_{j\in\N}$ be a decreasing sequence of closed linear
subspaces of $X$ and for each $j\in\N$, let $A_j$ and $B_j$ be dense
countable subsets of $X_j$. Then there is a Banach disk $D$ in $X$
such that for every $j\in\N$, both $A_j$ and $B_j$ are dense subsets
of the Banach space $(X_D\cap X_j,p_D)$.
\end{lemma}

\begin{proof} Let $d$ be a translation invariant metric defining the
topology of $X$. For each $j\in\N$, let $C_j$ be the set of all
linear combinations of the elements of $A_j\cup B_j$ with rational
coefficients. Obviously, $C_j$ is countable. Pick a map $f_j:\N\to
C_j$ such that $f_j^{-1}(x)$ is an infinite subset of $\N$ for every
$x\in C$. Since $A_j$ and $B_j$ are dense in $X_j$, we can find maps
$\alpha_j:\N\to A_j$ and $\beta_j:\N\to B_j$ such that
$d(2^m(f_j(m)-\alpha_j(m)),0)<2^{-j-m}$ and
$d(2^m(f_j(m)-\beta_j(m)),0)<2^{-j-m}$ for every $m\in\N$. Since
$A_j$ and $B_j$ are countable, we can write $A_j=\{x_{j,m}:m\in\N\}$
and $B_j=\{y_{j,m}:m\in\N\}$. Using metrizability of $X$, we can
find a sequence $\{\gamma_{j,m}\}_{m\in\N}$ of positive numbers such
that $d(\gamma_{j,m} x_{j,m},0)<2^{-j-m}$ and $d(\gamma_{j,m}
y_{j,n})<2^{-j-m}$ for every $m\in\N$. Enumerating the countable set
$$
\bigcup_{j,m\in\N}\{2^m(f_j(m)-\alpha_j(m)),2^m(f_j(m)-\beta_j(m)),\gamma_{j,m}
x_{j,m},\gamma_{j,m} y_{j,m}\}
$$
as one (convergent to 0) sequence and applying Lemma~\ref{l11} to
this sequence, we find that there is a Banach disk $D$ in $X$ such
that $X_D$ contains $A_j$ and $B_j$, the linear span of $A_j\cup
B_j$ is $p_D$-dense in $X_D\cap X_j$ and $f_j(m)-\alpha_j(m)\to 0$
and $f_j(m)-\beta_j(m)\to 0$ as $m\to\infty$ in $X_D$ for each
$j<n$. The $p_D$-density of the linear span of $A_j\cup B_j$ in
$X_D\cap X_j$ implies the $p_D$-density of $C_j$ in $X_D\cap X_j$.
Taking into account that $f_j^{-1}(x)$ is infinite for every $x\in
C_j$ and that $\alpha_j$ takes values in $A_j$, $p_D$-density of
$C_j$ in $X_D\cap X_j$ and the relation $\pd(f_j(m)-\alpha_j(m))\to
0$ imply that $A_j$ is $p_D$-dense in $X_D\cap X_j$. Similarly,
$B_j$ is $p_D$-dense in $X_D\cap X_j$. Thus $D$ satisfies all
required conditions.
\end{proof}

Applying Lemma~\ref{l22} with $X_j=X$, $A_j=A$ and $B_j=B$ for every
$j\in\N$, we obtain the following result.

\begin{lemma}\label{l22pr} Let $A$ and $B$ be dense countable subsets
of a Fr\'echet space $X$. Then there is a Banach disk $D$ in $X$
such that both $A$ and $B$ are dense subsets of the Banach space
$X_D$.
\end{lemma}

Applying the above lemma with $A=B$, we obtain the following result.

\begin{corollary}\label{l22a} Let $A$ be a dense countable subset
of a Fr\'echet space $X$. Then there is a Banach disk $D$ in $X$
such that $A$ is a dense subset of the Banach space $X_D$.
\end{corollary}

If $p$ is a continuous seminorm on a locally convex space $X$ and
$f\in X'$, we denote
$$
p^*(f)=\sup\{|f(x)|:x\in X,\ p(x)<1\}.
$$
Clearly $X'_p=\{f\in X':p^*(f)<\infty\}$ is a linear subspace of
$X'$ and $p^*$ is a norm on $X'_p$. It is easy to see that
$(X'_p,p^*)$ is a Banach space. The following result is Lemma~5.1 in
\cite{ss}.

\begin{lemma}\label{bububu} Let $p$ be a continuous seminorm on a
locally convex space $X$ and $E$ be a countably dimensional subspace
of $X$ such that $E\cap \ker p=\{0\}$. Then there exist a Hamel
basis $\{u_n\}_{n\in\N}$ in $E$ and a sequence $\{f_n\}_{n\in\N}$ in
$X'_p$ such that $f_n(u_m)=\delta_{n,m}$ for every $m,n\in\N$.
\end{lemma}

The following result features as Lemma~4.3 in \cite{ss}.

\begin{lemma}\label{lile} Let $X$ be a separable Fr\'echet space and
$p$ be a non-trivial continuous seminorm on $X$. Then for every
dense countable set $A\subset X$, there is $B\subseteq A$ such that
$B$ is $p$-independent and dense in $X$.
\end{lemma}

\begin{corollary}\label{nnn1} Let $X$ be a separable Fr\'echet space
and $p$ be a non-trivial continuous seminorm on $X$. Then there is a
dense countably dimensional subspace $E$ of $X$ such that $E\cap\ker
p=\{0\}$.
\end{corollary}

\begin{proof}Since $X$ is separable, there is a dense countable
subset $A$ of $X$. By Lemma~\ref{lile}, there is a dense in $X$
$p$-independent subset $B$ of $A$. Clearly $E=\spann(B)$ has all
desired properties.
\end{proof}

For a continuous seminorm $p$ on a locally convex space $X$, $f\in
X'$ and $x\in X$, we denote
$$
\widehat p(f,x)=p^*(f)p(x)
$$
with the agreement that $\infty\cdot 0=0$. That is, $\widehat
p(f,x)=0$ if $p(x)=0$, $\widehat p(f,x)=\infty$ if $p(x)>0$ and
$p^*(f)=\infty$ and $\widehat p(f,x)=p^*(f)p(x)$ otherwise.

\begin{lemma}\label{nuc} Let $X$ be a Fr\'echet space
$\{x_n\}_{n\in\N}$ and $\{f_n\}_{n\in\N}$ be sequences in $X$ and
$X'$ respectively and $\{p_n\}_{n\in\N}$ be a sequence of seminorms
on $X$ defining the topology of $X$. If for every $k\in\N$,
$c_k=\sum\limits_{n=1}^\infty\widehat p_k(f_n,x_n)<\infty$, then the
formula
$$
Tx=\sum_{n=1}^\infty f_n(x)x_n
$$
defines a continuous linear operator on $X$. Furthermore, if $c_k<1$
for every $k\in\N$, then the operator $I+T$ is invertible.
\end{lemma}

\begin{proof} It is easy to see that for every $x\in X$ and
$k,n\in\N$, $p_k(f_n(x)x_n)\leq p_k(x)\widehat p_k(f_n,x_n)$. It
follows that the series in the last display converges absolutely in
$X$ and $p_k(Tx)\leq c_kp_k(x)$ for every $x\in X$ and every
$k\in\N$. Hence $T$ is a well-defined continuous linear operator on
$X$.

Assume now that $c_k<1$ for every $k\in\N$. Define $S:X\to X$ by the
formula
$$
S=\sum_{n=0}^\infty (-T)^n.
$$
Since $p_k(Tx)\leq c_kp_k(x)$ for every $x\in X$ and every $k\in\N$,
we have $p_k(T^nx)\leq c_k^np_k(x)$ for every $x\in X$ and every
$k,n\in\N$. Since $c_k<1$, the series of operators in the above
display converges pointwise. By the uniform boundedness principle,
$S$ is a continuous linear operator. It is a routine exercise to
check that $S(I+T)=(I+T)S=I$. That is, $I+T$ is invertible.
\end{proof}

\begin{lemma}\label{BuBuBu} Let $X$ be a separable Fr\'echet space,
whose topology is given by an increasing sequence $\{p_n\}_{n\in\N}$
of seminorms and let $\{X_n\}_{n\in\Z_+}$ be a sequence of closed
linear subspaces of $X$ such that $X_{n+1}\subseteq X_{n}\cap\ker
p_{n+1}$ and $X_n/(X_{n}\cap\ker p_{n+1})$ is infinite dimensional
for each $n\in\Z_+$. Then there exist $\{u_{n,k}:n,k\in\Z_+\}\subset
X$ and $\{f_{n,k}:n,k\in\Z_+\}\subset X'$ such that
\begin{align}
&\text{$\spannn\{u_{m,k}:m\geq n,\ k\in\Z_+\}=X_n$ for every
$n\in\Z_+$};\label{BuBuBu1}
\\
&\text{for every $n\in\Z_+$, $u_{n,k}$ for $k\in\Z_+$ are
$p_{n+1}$-independent};\label{BuBuBu2}
\\
&\text{$p_{n+1}^*(f_{n,k})<\infty$ for every $n\in\N$ and
$k\in\Z_+$};\label{BuBuBu3}
\\
&\text{$f_{n,k}(u_{m,j})=\delta_{n,m}\delta_{k,j}$ for every
$n,k,m,j\in\Z_+$}.\label{BuBuBu4}
\end{align}
\end{lemma}

\begin{proof} Since $X_n/(X_{n}\cap\ker p_{n+1})$ is infinite dimensional,
$p_{n+1}$ is a non-trivial seminorm on the separable Fr\'echet space
$X_n$. By Lemma~\ref{lile}, there is a $p_{n+1}$-independent dense
sequence $\{v_{n,k}\}_{k\in\Z_+}$ in $X_n$. Fix a bijection
$\alpha=(\alpha_1,\alpha_2):\N\to\Z_+^2$. We shall construct
inductively $w_j\in X$ and $g_j\in X'$ for $j\in\N$ such that for
every $k\in\N$,
\begin{align}
&w_k\in v_{\alpha_1(k),\alpha_2(k)}+
\spann\{v_{\alpha_1(j),\alpha_2(j)}:j<k,\
\alpha_1(j)\geq\alpha_1(k)\};\label{BuBuBu5}
\\
&p^*_{\alpha_1(k)+1}(g_k)<\infty;\label{BuBuBu6}
\\
&g_k(w_k)=1;\label{BuBuBu7}
\\
&\text{$g_k(w_j)=0$ if $j<k$};\label{BuBuBu8}
\\
&\text{$g_j(w_k)=0$ if $j<k$}.\label{BuBuBu9}
\end{align}

We start by setting $w_1=v_{\alpha_1(1),\alpha_2(1)}$ and observing
that $w_1\notin\ker p_{n+1}$, where $n=\alpha_1(1)$. Indeed, this
follows from $p_{n+1}$-independence of $\{v_{n,k}\}_{k\in\Z_+}$. By
the Hahn--Banach theorem, there is $g_1\in X'$ such that
$p^*_{n+1}(g_1)<\infty$ and $g_1(w_1)=1$. Obviously,
(\ref{BuBuBu5}--\ref{BuBuBu9}) with $k=1$ are satisfied. Thus we
have our basis of induction.

Assume now that $m\geq 2$ and $w_j\in X$, $g_j\in X'$ for $j<m$
satisfy (\ref{BuBuBu5}--\ref{BuBuBu9}) for every $k<m$. We have to
construct $w_m$ and $g_m$ such that (\ref{BuBuBu5}--\ref{BuBuBu9})
hold for $k=m$. First, we define
\begin{equation}\label{foo}
w_m=v_{\alpha_1(m),\alpha_2(m)}-\sum_{j=1}^{m-1}
g_j(v_{\alpha_1(m),\alpha_2(m)})w_j.
\end{equation}
Let $n=\alpha_1(m)$ and $j<m$. By (\ref{BuBuBu6}) with $k=j$,
$p^*_{\alpha_1(j)+1}(g_j)<\infty$ and therefore $g_j$ vanishes on
$\ker p_{\alpha_1(j)+1}\supseteq X_{\alpha_1(j)+1}$. Since
$v_{\alpha_1(m),\alpha_2(m)}\in X_n$, we see that
$g_j(v_{\alpha_1(m),\alpha_2(m)})=0$ if $\alpha_1(j)<n=\alpha_1(m)$.
Thus (\ref{BuBuBu5}) for $k<m$ and (\ref{foo}) imply (\ref{BuBuBu5})
for $k=m$. The latter together with $p_{j+1}$-independence of
$\{v_{j,k}\}_{k\in\Z_+}$ ensures that the vectors $w_s$ for $s\leq
m$, $\alpha_1(s)\leq n$ are $p_{n+1}$-independent. By the
Hahn--Banach theorem, there exists $g_m\in X'$ such that
$p^*_{n+1}(g_m)<\infty$, $g_m(w_m)=1$ and $g_m(w_j)=1$ whenever
$j<m$ and $\alpha_1(j)\leq n$. If $j<m$ and $\alpha_1(j)>n$, then
$w_j\in X_{\alpha_1(j)}\subseteq X_{n+1}\subseteq \ker p_{n+1}$ and
therefore $g_m(w_j)=0$ since $g_m$ vanishes on $\ker p_{n+1}$. Thus
(\ref{BuBuBu6}), (\ref{BuBuBu7}) and (\ref{BuBuBu8}) with $k=m$ are
satisfied. Finally, using (\ref{foo}) and
(\ref{BuBuBu7}--\ref{BuBuBu9}) with $k<m$, one can easily verify
that (\ref{BuBuBu9}) for $k=m$ is also satisfied. This completes the
inductive construction of the sequences $\{w_n\}$ and $\{g_n\}$
satisfying (\ref{BuBuBu5}--\ref{BuBuBu9}).

Now we set $u_{\alpha_1(j),\alpha_2(j)}=w_j$ and
$f_{\alpha_1(j),\alpha_2(j)}=g_j$. Clearly (\ref{BuBuBu6}) implies
(\ref{BuBuBu3}), while (\ref{BuBuBu7}--\ref{BuBuBu9}) imply
(\ref{BuBuBu4}). Next, (\ref{BuBuBu2}) follows from
$p_{j+1}$-independence of $\{v_{j,k}\}_{k\in\Z_+}$ and
(\ref{BuBuBu5}). Finally, (\ref{BuBuBu5}) ensures that
$\spann\{u_{m,k}:m\geq n,\ k\in\Z_+\}=\spann\{v_{m,k}:m\geq n,\
k\in\Z_+\}$. It remains to recall that $\{v_{m,k}:m\geq n,\
k\in\Z_+\}$ is a dense subset of $X_n$ to see that (\ref{BuBuBu1})
follows.
\end{proof}

We also need the following lemma, which helps to verify the
completeness of locally convex spaces. It is a variation of a pretty
standard fact, see the completeness chapter in \cite{shifer}. We
include its proof for the sake of convenience.

\begin{lemma}\label{comple} Let $(E,\tau)$ be a complete locally convex space, $F$
be a linear subspace of $E$ and ${\cal Q}$ be a collection of
seminorms on $F$ defining a locally convex topology $\theta$ on $F$
stronger than the one inherited from $E:$ \
$\theta\supseteq\tau\bigr|_F$. Assume also that for every $q\in{\cal
Q}$, the ball $B_q=\{x\in F:q(x)\leq 1\}$ is $\tau\bigr|_F$-closed
in $F$ and that for every $\theta$-Cauchy net $\{x_\alpha\}$ in $F$
its $\tau$-limit belongs to $F$: $\lim_\tau x_\alpha\in F$. Then
$(F,\theta)$ is complete.
\end{lemma}

\begin{proof} It suffices to prove that every $\tau$-convergent to $0$
$\theta$-Cauchy net $\{x_\alpha\}$ in $F$ is also
$\theta$-convergent to $0$. Indeed, due to our assumptions, this
will imply that every $\theta$-Cauchy net $\{x_\alpha\}$ in $F$ is
$\theta$-convergent to $\lim_\tau x_\alpha\in F$.

Let $\{x_\alpha\}_{\alpha\in D}$ be a $\theta$-Cauchy net
$\{x_\alpha\}$ in $F$ such that $\lim_\tau x_\alpha=0$. Take any
$q\in{\cal Q}$. Since $\{x_\alpha\}$ is $\theta$-Cauchy, for every
$\epsilon>0$, there is $\alpha_0\in D$ such that
$q(x_\alpha-x_\beta)\leq \epsilon$ for every $\alpha,\beta\in D$,
$\alpha,\beta>\alpha_0$. Fix temporarily $\beta\in D$ such that
$\beta>\alpha_0$. Then we can write
$$
\{x_\alpha:\alpha>\alpha_0\}\subseteq W=\{x\in
F:q(x-x_\beta)\leq\epsilon\}.
$$
Since $B_q$ is $\tau$-closed in $F$, $W$ is also $\tau$-closed in
$F$. Since $\{x_\alpha\}$ is $\tau$-convergent to $0$, from the
above display it follows that $0\in W$. Hence
$q(x_\beta)\leq\epsilon$. Since $\beta\in D$ is an arbitrary element
such that $\beta>\alpha_0$, we have $q(x_\beta)\leq\epsilon$ for
every $\beta>\alpha_0$. Hence $q(x_\alpha)\to 0$. Since $q\in{\cal
Q}$ is arbitrary, $\lim_\theta x_\alpha=0$, as required.
\end{proof}

\subsection{An obstacle to embedding of Fr\'echet spaces \label{FFFUUU}}

In the proof of Theorem~\ref{grgr11}, we will need to establish that
a certain Fr\'echet space is non-isomorphic to a subspace of another
given Fr\'echet space. Our approach is pretty standard. If $p$ is a
seminorm on a vector space $X$, $q$ is a seminorm on a vector space
$Y$, symbol $L_{p,q}(X,Y)$ stands for the space of linear maps
$T:X\to Y$ for which there is $c>0$ such that $q(Tx)\leq cp(x)$ for
every $x\in X$. The minimal possible $c$ with this property will be
denoted $\pi_{p,q}(T)$. For $n\in\Z_+$ and $T\in L_{p,q}(X,Y)$, we
denote
$$
\alpha_n(p,q,T)=\inf\{\pi_{p,q}(T-S):S\in L_{p,q}(X,Y),\ \dim
S(X)\leq n\}.
$$
The numbers $\alpha_n(p,q,T)$ are usually called approximation
numbers. Clearly, they form a decreasing sequence of non-negative
numbers.

In this section, the symbol $\Omega$ stands for the set of
decreasing sequences $r=\{r_n\}_{n\in\Z_+}$ of positive numbers such
that $\{nr_n\}$ is bounded. For $r\in\Omega$, the symbol ${\cal
F}_{[r]}$ stands for the class of $X\in{\cal F}$ such that for every
continuous seminorm $p$ on $X$ and $j\in\N$, there is a continuous
seminorm $q$ on $X$ satisfying $q\geq p$ and
$$
\lim_{n\to\infty}\frac{n^j\alpha_n(q,p,{\rm Id}_X)}{r_n}=0.
$$
It is an easy exercise to see that ${\cal F}_{[r]}$ is a subclass of
the class ${\cal N}$ of nuclear Fr\'echet spaces and ${\cal
F}_{[r]}={\cal N}$ if $n^{-k}=O(r_n)$ for some $k\in\N$. It is easy
(see \cite{pich}) to verify that if $Z$ is a subspace of $Y$ and
$T\in L_{p,q}(X,Y)$ satisfies $T(X)\subseteq Z$, then
$\alpha_n(p,q,T)\leq \alpha_n(p,q,\widetilde{T})\leq
(n+1)\alpha_n(p,q,T)$, where $\widetilde{T}$ denotes $T$ considered
as an element of $L_{p,q}(X,Z)$. It follows that each of the classes
${\cal F}_{[r]}$ is closed under passing to a closed linear
subspace. We will not need this, but it is worth noting that each
${\cal F}_{[r]}$ is also closed under quotients.

\begin{lemma}\label{fufu} Let $\{X_n\}_{n\in\N}$ be a sequence in
${\cal F}\setminus{\cal F}_0$. Then there is $r\in\Omega$ such that
$X_n\notin {\cal F}_{[r]}$ for every $n\in\N$.
\end{lemma}

\begin{proof} Since $X_n$ is non-isomorphic to $\omega$, there is an
increasing sequence $\{p_{n,m}\}_{m\in\N}$ of seminorms on $X_n$
defining the topology of $X_n$ such that $X_n/\ker p_{n,1}$ is
infinite dimensional. Then $\alpha_t(p_{n,k},p_{n,m},{\rm
Id}_{X_n})>0$ for every $t\in\Z_+$ and $k,m,n\in\N$ satisfying
$k<m$. Now we can pick $r\in\Omega$ such that $nr_n\to 0$ and
$$
\lim_{t\to\infty}\frac{r_t}{\alpha_t(p_{n,m},p_{n,k},{\rm
Id}_{X_n})}=0\ \ \text{whenever $k,m,n\in\N$ and $k<m$.}
$$
Since $\{p_{n,m}\}_{m\in\N}$ defines the topology of $X_n$, we have
$X_n\notin{\cal F}_{[r]}$.
\end{proof}

\begin{lemma}\label{cuddly} Let $r\in\Omega$, $X\in{\cal F}$ and
$\{p_n\}_{n\in\N}$ be an increasing  sequence of seminorms defining
the topology of $X$ such that $\ker p_n/\ker p_{n+1}$ is infinite
dimensional for each $n\in\N$. Then there exists a closed linear
subspace $Y$ of $\ker p_1$ such that $Y\in {\cal F}_{[r]}$ and
$Y\cap \ker p_n$ is infinite dimensional for every $n\in\N$.
\end{lemma}

\begin{proof}Denote $X_n=\ker p_n$ for $n\in\N$ and $X_0=X$.
By Lemma~\ref{BuBuBu}, there exist $\{u_{n,k}:n,k\in\Z_+\}\subset X$
and $\{f_{n,k}:n,k\in\Z_+\}\subset X'$ such that
(\ref{BuBuBu1}--\ref{BuBuBu4}) are satisfied. By
Corollary~\ref{l22a}, there exists a Banach disk $D$ in $X$ such
that $u_{n,k}\in X_D$ for every $n,k\in\Z_+$. For each $n\in\N$, let
$q_n$ be the Minkowski functional of the set $D+\{x\in X:p_n(x)\leq
1\}$. It is easy to see that each $q_n$ is equivalent to $p_n$ and
therefore $\{q_n\}$ is an increasing sequence of seminorms defining
the topology of $X$. It also has an extra convenient property:
$q_n(x)\leq p_D(x)$ for every $x\in X$. Let $\beta:\N\times\N\to\N$
be a bijection. Now for $A\subseteq\N\times\N$, we can define a
linear map $S_A:X\to X$ by the formula
$$
S_Ax=\sum_{(n,k)\in A}
\frac{r_{\beta(n,k)}f_{n+1,2k}(x)}{2^{\beta(n,k)}
q_{n+2}^*(f_{n+1,2k})p_D(u_{n,k})}u_{n,k}.
$$
We also set $S=S_{\N\times\N}$. First, we shall verify that each
$S_A$ is a well-defined continuous linear operator. To this end,
consider the linear map $S_{A,n}:X\to X$ given by
$$
S_{A,n}x=\sum_{k:(n,k)\in A}
\frac{r_{\beta(n,k)}f_{n+1,2k}(x)}{2^{\beta(n,k)}
q_{n+2}^*(f_{n+1,2k})p_D(u_{n,k})}u_{n,k}.
$$
Since for each $x\in X$, $|f_{n+1,2k}(x)|\leq
q_{n+2}(x)q_{n+2}^*(f_{n+1,2k})$. Plugging this estimate into the
above display, we immediately see that the above series converges
absolutely in the Banach space $X_D$ and therefore in $X$. It
follows that $S_{A,n}$ is well-defined. Furthermore
$$
p_D(S_{A,n}x)\leq q_{n+2}(x)\sum_{k:(n,k)\in A}
\frac{r_{\beta(n,k)}}{2^{\beta(n,k)}}\leq2\cdot
2^{-j(A,n)}r_{j(A,n)}q_{n+2}(x),
$$
where $j(A,n)=\min\{s\in\N:\beta^{-1}(s)\in
A\cap(\{n\}\times\Z_+)\}$. It immediately follows that  $S_{A,n}$
are continuous. Furthermore, $S_{A,n}(X)\subseteq X_n$. Since $q_n$
vanishes on $X_n$ it follows that the series
$S_Ax=\sum\limits_{n=1}^\infty S_{A,n}x$ converges absolutely in the
Fr\'echet space $X$. Thus $S_A$ is well defined. The continuity of
each $S_{A,n}$ and the uniform boundedness principle imply the
continuity of $S_A$. Moreover, for $x\in X$ and $m\in\N$, we can use
the above display to see that
\begin{align}
q_m(S_Ax)&=q_m\biggl(\sum_{n=0}^\infty S_{A,n}x\biggr)=
q_m\biggl(\sum_{n=0}^{m-1} S_{A,n}x\biggr)\leq\sum_{n=0}^{m-1}
q_m(S_{A,n}x)\leq \sum_{n=0}^{m-1} p_D(S_{A,n}x)\notag
\\
&\leq 2 \sum_{n=0}^{m-1} 2^{-j(A,n)}r_{j(A,n)}q_{n+2}(x)
\leq2q_{m+1}(x) \sum_{n=0}^{m-1} 2^{-j(A,n)}r_{j(A,n)}\leq 4
q_{m+1}(x)2^{-t(A)}r_{t(A)},\label{truu}
\end{align}
where $t(A)=\min\{s\in\N:\beta^{-1}(s)\in A\}$.

Observe that according to (\ref{BuBuBu1}--\ref{BuBuBu4}),
$S(X_{n+1})\subseteq X_n$ for $n>1$ and $S$ vanishes on $X_1$. Now
we define $Y$ as the closure in $X$ of
$$
G=\Bigl\{\sum_{n=1}^\infty x_n:x_n\in X_n\ \ \text{and}\ \
Sx_{n+1}=x_n\ \ \text{for}\ \ n\in\N\Bigr\}.
$$
Note that each series in the definition of $G$ converges absolutely
since $x_n\in X_n$ and each $q_k$ vanishes on $X_s$ for $s\geq k$.
First, we shall show that $G\cap X_n$ is infinite dimensional for
every $n\in\N$. Let $n\in\N$ and $m$ be an odd positive integer.
Using (\ref{BuBuBu1}--\ref{BuBuBu4}) and the definition of $S$ we
see that $Su_{n,m}=0$ and for each $j\in\Z_+$ there is $c_j>0$ such
that $c_0=1$ and $S(c_{j+1}u_{n+j+1,2^{j+1}m})=c_{j}u_{n+j,2^jm}$.
Since $u_{n,k}\in X_n$, it follows that
$w_{n,m}=\sum\limits_{j=0}^\infty c_{j}u_{n+j,2^jm}\in G$. Using
(\ref{BuBuBu1}--\ref{BuBuBu4}) once again, we see that
$f_{n,m}(w_{n,m'})=\delta_{m,m'}$ whenever $m$ and $m'$ are odd
positive integers. Thus $\{w_{n,m}:m+1\in 2\N\}$ is an infinite
linearly independent subset of $G\cap X_n$. Since $G\subseteq Y$,
$Y\cap X_n$ is infinite dimensional.

Now let $m\geq 2$ and $x=\sum\limits_{n=1}^\infty x_n\in G$, where
$x_n\in X_n$ and $Sx_{n+1}=x_n$ for $n\in\N$. Let
$u=\sum\limits_{n=1}^m x_n$. Using the equalities $Sx_{n+1}=x_n$ and
$Sx_1=0$ as well as the inclusions $x_k\in X_k$, we obtain
$$
x-u\in\ker p_{m+1}\ \ \ \text{and}\ \ \ x-Su\in\ker p_m.
$$
It follows that up to moving along the kernels of the corresponding
seminorms, ${\rm Id}_G$ coincides with $S\bigr|_G$ as elements of
$L_{p_{m+1},p_m}(G,G)=L_{q_{m+1},q_m}(G,G)$. Since $G$ is dense in
$Y$, we have
$$
\alpha_k(q_{m+1},q_m,{\rm Id}_Y)=\alpha_k(q_{m+1},q_m,S\bigr|_Y)\ \
\ \text{for each $k\in\Z_+$}.
$$
Now for each $k\in\Z_+$, we can consider
$A_k=\N\times\N\setminus\{(n,j)\in\N\times\N:\beta(n,j)\leq k\}$.
Obviously $|\N\times\N\setminus A_k|=k$ and $t(A_k)=k+1$. By
\ref{truu}, $q_m(S_{A_k}x)\leq 2^{1-k}r_{k+1}q_{m+1}(x)$ for every
$x\in X$ (and therefore for every $x\in Y$). Since the linear
operator $S-S_{A_k}$ has rank $k$, the above display yields
$$
\alpha_k(q_{m+1},q_m,{\rm Id}_Y)=\alpha_k(q_{m+1},q_m,S\bigr|_Y)\leq
2^{1-k}r_{k+1}.
$$
Since the sequence $\{q_m\}$ defines the topology of $Y$, $Y\in
{\cal F}_{[r]}$.
\end{proof}

\subsection{A specific countable topological space}

We call a sequence trivial if it is eventually stabilizing.

\begin{lemma}\label{m2} There exists a regular topology $\tau$ on $\Z$ such
that the topological space $\Z_\tau=(\Z,\tau)$ has the following
properties$:$
\begin{itemize}\itemsep-2pt
\item[\rm (a)]$f:\Z\to \Z$, $f(n)=n+1$ is a homeomorphism of $\Z_\tau$ onto
itself$\,;$
\item[\rm (b)]$\Z_+$ is dense in $\Z_\tau;$
\item[\rm (c)]for every $z\in\C\setminus\{0,1\}$,
$n\mapsto z^n$ is non-continuous as a map from $\Z_\tau$ to $\C;$
\item[\rm (d)]$\Z_\tau$ has no non-trivial convergent sequences.
\end{itemize}
\end{lemma}

\begin{proof}Consider the Hilbert space $\ell_2(\Z)$ and the
bilateral weighted shift $T\in L(\ell_2(\Z))$ given by
$Te_n=e_{n-1}$ if $n\leq 0$ and $Te_n=2e_{n-1}$ if $n>0$, where
$\{e_n\}_{n\in\Z}$ is the canonical orthonormal basis in
$\ell_2(\Z)$. Symbol $\hh_\sigma$ stands for $\ell_2(\Z)$ equipped
with its weak topology $\sigma$. Clearly, $T\in GL(\ell_2(\Z))$ and
therefore $T\in GL(\hh_\sigma)$. According to Chan and Sanders
\cite{cs}, there is $x\in\ell_2(\Z)$ such that the set
$O=\{T^nx:n\in\Z_+\}$ is dense in $\hh_\sigma$. Then
$Y=\{T^nx:n\in\Z\}$ is also dense in $\hh_\sigma$. We equip $Y$ with
the topology inherited from $\hh_\sigma$ and transfer it to $\Z$ by
declaring the bijection $n\mapsto T^nx$ from $\Z$ to $Y$ a
homeomorphism.

Since $\sigma$ is a completely regular topology, so is the just
defined topology $\tau$ on $\Z$. Since $T\in GL(\hh_\sigma)$ and $Y$
is a subset of $\ell_2(\Z)$ invariant for both $T$ and $T^{-1}$, $T$
is a homeomorphism from $Y$ onto itself. Since $T(T^nx)=T^{n+1}x$,
it follows that $f$ is a homeomorphism of $\Z_\tau$ onto itself.
Density of $O$ in $\hh_\sigma$ implies the density of $\Z_+$ in
$\Z_\tau$.

Observe that the sequence $\{\|T^nx\|\}_{n\in\Z}$ is strictly
increasing and $\|T^nx\|\to\infty$ as $n\to+\infty$. Indeed, the
inequality $\|Tu\|\geq \|u\|$ for $u\in\ell_2(\Z)$ follows from the
definition of $T$. Hence $\{\|T^nx\|\}_{n\in\Z}$ is increasing.
Assume that $\|T^{n+1}x\|=\|T^nx\|$ for some $n\in\Z$. Then, by
definition of $T$, $T^nx$ belongs to the closed linear span $L$ of
$e_n$ with $n<0$. The latter is invariant for $T$ and therefore
$T^mx\in L$ for $m\geq n$, which is incompatible with the
$\sigma$-density of $O$. Next, if $\|T^nx\|$ does not tend to
$\infty$ as $n\to+\infty$, the sequence $\{\|T^nx\|\}_{n\in\Z_+}$ is
bounded. Since every bounded subset of $\ell_2(\Z)$ is
$\sigma$-nowhere dense, we have again obtained a contradiction with
the $\sigma$-density of $O$.

In order to show that $X$ has no non-trivial convergent sequences,
it suffices to show that $Y$ has no non-trivial convergent
sequences. Assume that $\{T^{n_k}x\}_{k\in\Z_+}$ is a non-trivial
convergent sequence in $Y$. Without loss of generality, we can
assume that the sequence $\{n_k\}$ of integers is either strictly
increasing or strictly decreasing. If $\{n_k\}$ is strictly
increasing the above observation ensures that
$\|T^{n_k}x\|\to\infty$ as $k\to\infty$. Since every
$\sigma$-convergent sequence is bounded, we have arrived to a
contradiction. If $\{n_k\}$ is strictly decreasing, then by the
above observation, the sequence $\{\|T^{n_k}x\|\}$ of positive
numbers is also strictly decreasing and therefore converges to
$c\geq 0$. Then $\|T^lx\|>c$ for every $l\in\Z$. Since
$\{T^{n_k}x\}$ $\sigma$-converges to $T^mx\in Y$, the upper
semicontinuity of the norm function with respect to $\sigma$ implies
that $\|T^mx\|\leq c$ and we have arrived to a contradiction.

Finally, let $z\in\C\setminus\{0,1\}$ and $f:\Z\to \C$, $f(n)=z^n$.
It remains to show that $f$ is not continuous as a function on
$\Z_\tau$. Equivalently, it is enough to show that the function
$g:Y\to\C$, $g(T^nx)=z^n$ is non-continuous. Assume the contrary.
First, consider the case $|z|\neq 1$. In this case the the topology
on the set $M=\{z^n:n\in\Z\}$ inherited from $\C$ is the discrete
topology. Continuity of the bijection $g:Y\to M$ implies then that
$Y$ is also discrete, which is not the case: the density of $Y$ in
$\hh_\sigma$ ensures that $Y$ has no isolated points. It remains to
consider the case $|z|=1$, $z\neq1$. In this case the closure $G$ of
$\{z^n:n\in \Z\}$ is a closed subgroup of the compact abelian
topological group $\T$. Since $x$ is a hypercyclic vector for $T$
acting on $\hh_\sigma$ and $z$ generates the compact abelian
topological group $G$, \cite[Corollary~4.1]{66} implies that
$\{(T^nx,z^n):n\in\Z_+\}$ is dense in $\hh_\sigma\times G$. Hence
the graph of $g$ is dense in $Y\times G$. Since $G$ is not a
singleton, the latter is incompatible with the continuity of $g$.
The proof is complete.
\end{proof}

\begin{remark}\label{uuu}
It is easy to see that the topology $\tau$ on $\Z$ constructed in
the proof of Lemma~\ref{m2} does not agree with the group structure.
That is $\Z_\tau$ is not a topological group. Indeed, it is easy to
see that the group operation $+$ is only separately continuous on
$\Z_\tau$, but not jointly continuous. As a matter of curiosity, it
would be interesting to find out whether there exists a topology
$\tau$ on $\Z$ satisfying all conditions of Lemma~\ref{m2} and
turning $\Z$ into a topological group.
\end{remark}

\section{Isomorphisms of countably dimensional spaces in ${\cal M}$}

The following theorem is Theorem~3.1 in \cite{ss}.

\begin{theorem}\label{omeg} The group $GL(\omega)$ acts transitively on
the set ${\cal E}(\omega)$ of dense countably dimensional subspaces
of $\omega$.
\end{theorem}

\begin{corollary}\label{isisis} Every two countably dimensional
spaces in ${\cal M}_0$ are isomorphic.
\end{corollary}

\begin{proof} Let $X$ and $Y$ be two countably dimensional
spaces in ${\cal M}_0$. Since $X$ and $Y$ are metrizable and sport
weak topologies, the completions of both $X$ and $Y$ are isomorphic
to $\omega$. Thus, without loss of generality, both $X$ and $Y$ are
dense countably dimensional linear subspaces of $\omega$. It remains
to apply Theorem~\ref{omeg}.
\end{proof}

\begin{lemma}\label{icod} Let $X$ be a Fr\'echet space isomorphic to $\omega$,
$E$ and $F$ be countably dimensional subspaces of $X$ such that both
$X/\overline{E}$ and $X/\overline{F}$ are infinite dimensional. Then
there is $T\in GL(X)$ such that $T(E)=F$.
\end{lemma}

\begin{proof}Since, up to an isomorphism, there is only one infinite
dimensional Fr\'echet space whose topology is weak (it is $\omega$),
each of the spaces $\overline{E}$, $\overline{F}$,
$\omega/\overline{E}$ and $\omega/\overline{F}$ is isomorphic to
$\omega$. Since every closed linear subspace of $\omega$ is
complemented \cite{bonet}, $X=\overline{E}\oplus
G=\overline{F}\oplus H$, where $G$ and $H$ are isomorphic to
$\omega$. Thus we can, without loss of generality, assume that
$X=\omega\times\omega$ and $E$ and $F$ are dense (countably
dimensional) subspaces of $\omega\times\{0\}$. By
Theorem~\ref{omeg}, there is $R\in GL(\omega)$ such that $(R\oplus
I)(E)=F$. Hence $R\oplus I$ is a required isomorphism.
\end{proof}

\subsection{Isomorphisms of spaces in ${\cal M}_1$}

The following theorem, generalizing the main result of \cite{gri2},
is Theorem~1.5 in \cite{ss}.

\begin{theorem}\label{GR} Let $X$ be a Fr\'echet space, $p$ be a non-trivial
continuous seminorm on $X$, and $A$ and $B$ be two countable dense
subsets of $X$ such that both $A$ and $B$ are $p$-independent. Then
there exists $R\in GL(X)$ such that $R(A)=B$ and $Rx=x$ for every
$x\in \ker p$.
\end{theorem}

\begin{corollary}\label{GR2} Let $E,F\in{\cal M}_1$ be countably
dimensional. Then $E$ and $F$ are isomorphic if and only if their
completions $\overline{E}$ and $\overline{F}$ are isomorphic.
\end{corollary}

\begin{proof} If $T:E\to F$ is an isomorphism, then the extension of
$T$ by continuity provides an isomorphism of $\overline{E}$ and
$\overline{F}$. Conversely, assume that $\overline{E}$ and
$\overline{F}$ are isomorphic. Then, without loss of generality, we
can assume that both $E$ and $F$ are dense countably dimensional
subspaces of the same Fr\'echet space $X$. Since $E,F\in{\cal M}_1$,
there exist continuous seminorms $p_1$ and $p_2$ on $X$ such that
$E\cap\ker p_1=F\cap\ker p_2=\{0\}$. Hence $E\cap\ker p=F\cap\ker
p=\{0\}$, where $p=p_1+p_2$. By Lemma~\ref{lile}, there are
$A_1\subset E$ and $B_1\subset F$ such that both $A_1$ and $B_1$ are
$p$-independent and dense in $X$. Pick a dense Hamel basis $A$ in
$E$ and a dense Hamel basis $B$ in $F$ such that $A_1\subseteq A$
and $B_1\subseteq B$. Since $E\cap\ker p=F\cap\ker p=\{0\}$, $A$ and
$B$ are $p$-independent. Moreover, $A$ and $B$ are dense in $X$ and
countable. By Theorem~\ref{GR}, there is $R\in GL(X)$ such that
$R(A)=B$. Obviously, $R(E)=F$. Hence $R\bigr|_E:E\to F$ is an
isomorphism.
\end{proof}

\subsection{Isomorphisms of spaces in ${\cal M}_3$}

The following result features as Lemma~2.2 in \cite{ss}.

\begin{lemma}\label{step1}Let $\epsilon>0$, $X$ be a locally convex
space, $D$ be a Banach disk in $X$, $Y$ be a closed linear subspace
of $X$, $\Lambda\subseteq Y\cap X_D$ be a dense subset of $Y$ such
that $\Lambda$ is $p_D$-dense in $Y\cap X_D$, $p$ be a continuous
seminorm on $X$, $L$ be a finite dimensional subspace of $X$ and
$T\in L(X)$ be a finite rank operator such that $T(Y)\subseteq Y\cap
X_D$, $T(\ker p)\subseteq\ker p$ and $\ker (I+T)=\{0\}$. Then
\begin{itemize}
\item[{\rm (1)}]for every $u\in Y\cap X_D$ such that
$(u+L)\cap \ker p=\varnothing$, there are $f\in X'$ and $v\in Y\cap
X_D$ such that $p^*(f)=1$, $f\bigr|_L=0$, $p_D(v)<\epsilon$,
$(I+R)u\in \Lambda$ and $\ker(I+R)=\{0\}$, where $Rx=Tx+f(x)v;$
\item[{\rm (2)}]for every $u\in Y\cap X_D$ such that
$(u+(I+T)(L))\cap\ker p=\varnothing$, there are $f\in X'$, $a\in
\Lambda$ and $v\in Y\cap X_D$ such that $p^*(f)=1$, $f\bigr|_L=0$,
$p_D(v)<\epsilon$, $(I+R)a=u$ and $\ker(I+R)=\{0\}$, where
$Rx=Tx+f(x)v.$
\end{itemize}
\end{lemma}

\begin{theorem}\label{GRGRGR} Let $X$ be a Fr\'echet space, whose
topology is defined by an increasing sequence $\{p_n\}_{n\in\N}$ of
seminorms. Assume also that $\{X_n\}_{n\in\Z_+}$ is a sequence of
closed infinite dimensional linear subspaces of $X$ such that
$X_0=X$ and $X_n\subseteq X_{n-1}\cap\ker p_n$ for every $n\in\N$.
For each $n\in\Z_+$ let $A_n$ and $B_n$ be $p_{n+1}$-independent
countable dense subsets of $X_n$. Then there exists $R\in GL(X)$
such that $R(A_n)=B_n$ for every $n\in\Z_+$.
\end{theorem}

\begin{proof} By Lemma~\ref{l22}, we can find a Banach disk $D$ in
$X$ such that for every $j\in\Z_+$, both $A_j$ and $B_j$ are dense
subsets of the Banach space $(X_D\cap X_j,p_D)$. Let $q_n$ be the
Minkowski functional of $D\cup\{x\in X:p_n(x)\leq 1\}$. It is easy
to see that each $q_n$ is a seminorm equivalent to $p_n$. Hence
$\{q_n\}$ is another increasing sequence of seminorms on $X$
defining the topology of $X$ and $\ker p_n=\ker q_n$ for $n\in\N$.
In particular, $X_n\subseteq X_{n-1}\cap\ker q_n$ for every $n\in\N$
and $A_n$ and $B_n$ are $q_{n+1}$-independent countable dense
subsets of $X_n$ for $n\in\Z_+$. The point of introducing the
seminorms $q_n$ is their extra property:
$$
q_n(x)\leq p_D(x)\ \ \text{for every $n\in\N$ and $x\in X_D$.}
$$

First, observe that the sets $A_j$ are pairwise disjoint. Indeed,
let $j<l$, $x\in A_j$ and $y\in A_l$. Since $A_l\subseteq
X_l\subseteq \ker p_l$, $p_l(y)=0$. Since $A_j$ is
$p_{j+1}$-independent $p_l(x)\geq p_{j+1}(x)>0$. Thus $x\neq y$ and
therefore $A_j$ are pairwise disjoint. Similarly, $B_j$ are pairwise
disjoint. Hence for $A=\bigcup\limits_{j=0}^\infty A_j$ and
$B=\bigcup\limits_{j=0}^\infty B_j$, the maps
$$
\text{$\tau:A\to\Z_+$, $\tau(a)=j$ if $a\in A_j$ and
$\sigma:B\to\Z_+$, $\sigma(b)=j$ if $b\in B_j$}
$$
are well-defined. Fix bijections $\alpha:\N\to A$ and $\beta:\N\to
B$ and a sequence $\{\epsilon_k\}_{k\in\N}$ of positive numbers
satisfying $\sum\limits_{k=1}^\infty\epsilon_k<1$.

We shall construct inductively sequences $\{s_j\}_{j\in\N}$ and
$\{m_j\}_{j\in\N}$ of natural numbers, $\{v_j\}_{j\in\N}$ in $X$,
$\{f_j\}_{j\in\N}$ in $X'$ and $\{T_k\}_{k\in\Z_+}$ in $L(X)$ such
that $T_0=0$ and for every $k\in\N$,
\begin{align}
&T_kx=\sum_{j=1}^k f_j(x)v_j;\label{ttt0}
\\
&m_j\neq m_l\ \ \text{and}\ \ s_j\neq s_l\ \ \text{for $1\leq
j<l\leq k$};\label{ttt1}
\\
&\tau(\alpha(m_j))=\sigma(\beta(s_j))\ \ \text{for $1\leq j\leq
k$};\label{ttt2}
\\
&\{1,\dots,j\}\subseteq\{m_1,\dots,m_{2j}\}\cap\{s_1,\dots,s_{2j}\}\
\ \text{whenever $2j\leq k$};\label{ttt3}
\\
&(I+T_k)\alpha(m_j)=\beta(s_j)\ \ \text{for $1\leq j\leq k$};
\label{ttt4}
\\
&\text{$v_j\in X_{r_j}\cap X_D$, $p_D(v_j)<\epsilon_j$ and
$q_{r_j+1}^*(f_j)=1$ for $1\leq j\leq k$, where
$r_j=\tau(\alpha(m_j))$}. \label{ttt5}
\end{align}

$T_0=0$ serves as the basis of induction. Assume now that $n\in\Z_+$
and that $s_j$, $m_j$, $v_j$, $f_j$ and $T_j$ for $j\leq 2n$
satisfying (\ref{ttt0}--\ref{ttt5}) with $k\leq 2n$ are already
constructed. We shall construct $s_j$, $m_j$, $v_j$, $f_j$ and $T_j$
with $j\in\{2n+1,2n+2\}$ such that (\ref{ttt0}--\ref{ttt5}) hold for
every  $k\leq 2n+2$.

Let $j\leq 2n$ and $r\in\N$. By (\ref{ttt5}), $v_j\in X_{r_j}$ with
$r_j=\tau(\alpha(m_j))$. Since $q_r$ vanishes on $X_r$, $q_r(v_j)=0$
and therefore $\widehat q_r(f_j,v_j)=0$ if $r\leq r_j$. On the other
hand, if $r>r_j$, then $q_{r}^*(f_j)\leq q_{r_j+1}^*(f_j)=1$ and
$q_r(v_j)\leq p_D(v_j)<\epsilon_j$. Hence $\widehat
q_r(f_j,v_j)<\epsilon_j$ if $r\leq r_j$. Thus in any case
\begin{equation}\label{TT0}
\widehat q_r(f_j,v_j)<\epsilon_j\ \ \text{for every $r\in\N$}.
\end{equation}
The estimate (\ref{TT0}) and Lemma~\ref{nuc} together with the
inequality $\sum\limits_{k=1}^\infty\epsilon_k<1$ and formula
(\ref{ttt0}) ensure that $I+T_{2n}$ is invertible. Furthermore, if
$r\leq r_j$, then $v_j\in X_r\subseteq \ker q_r$, while if $r>r_j$,
then $q^*_r(f_j)$ is finite and therefore $f_j$ vanishes on $\ker
q_r\supseteq X_r$. It follows that
\begin{equation}\label{TT1}
\text{$X_r$ and $\ker q_r$ are invariant for $T_j$ for every
$r\in\N$ and each $j$.}
\end{equation}
Since each $v_j$ belongs to $X_D$, each $T_j$ takes values in $X_D$.
Now we define
\begin{equation}\label{TT2}
m_{2n+1}=\min(\N\setminus\{m_j:j\leq 2n\})\ \ \text{and}\ \
r=\tau(\alpha(m_{2n+1})).
\end{equation}
It is a routine exercise to see that all conditions of the first
part of Lemma~\ref{step1} with $u=\alpha(m_{2n+1})$, $p=q_{r+1}$,
$Y=X_r$, $L=\spann\{\alpha(m_j):j\leq 2n\}$, $T=T_{2n}$,
$\epsilon=\epsilon_{2n+1}$ and
$\Lambda=B_r\setminus\{\beta(s_j):j\leq 2n\}$ are satisfied. Thus
Lemma~\ref{step1} provides $f_{2n+1}\in X'$ and $v_{2n+1}\in X_D\cap
X_r$ such that $q_{r+1}^*(f_{2n+1})=1$, $f_{2n+1}\bigr|_L=0$,
$p_D(v_{2n+1})<\epsilon_{2n+1}$ and $(I+T_{2n+1})u\in \Lambda$,
where $T_{2n+1}x=T_{2n}x+f_{2n+1}(x)v_{2n+1}$. The inclusion
$(I+T_{2n+1})u\in \Lambda$ means that
$(I+T_{2n+1})u=\beta(s_{2n+1})$ with $s_{2n+1}\in
\N\setminus\{s_j:j\leq 2n\}$. Since $u=\alpha(m_{2n+1})$ and
$f_{2n+1}\bigr|_L=0$, (\ref{ttt4}) with $k=2n$ implies that
(\ref{ttt4}) with $k=2n+1$ is also satisfied. The equality
$\tau(\alpha(m_{2n+1}))=\sigma(\beta(s_{2n+1}))=r$ gives
(\ref{ttt2}) for $k=2n+1$. The properties (\ref{ttt0}), (\ref{ttt1})
and (\ref{ttt5}) for $k=2n+1$ are satisfied by construction. We
postpone the discussion of (\ref{ttt3}) for later.

Define
\begin{equation}\label{TT3}
s_{2n+2}=\min(\N\setminus\{s_j:j\leq 2n+1\})\ \ \text{and}\ \
r'=\sigma(\beta(s_{2n+2})).
\end{equation}
It is a routine exercise to see that all conditions of the second
part of Lemma~\ref{step1} with $u=\beta(s_{2n+2})$, $p=q_{r'+1}$,
$Y=X_{r'}$, $L=\spann\{\alpha(m_j):j\leq 2n+1\}$, $T=T_{2n+1}$,
$\epsilon=\epsilon_{2n+2}$ and
$\Lambda=A_{r'}\setminus\{\alpha(m_j):j\leq 2n+1\}$ are satisfied.
Thus Lemma~\ref{step1} provides $f_{2n+2}\in X'$, $a\in\Lambda$ and
$v_{2n+2}\in X_D\cap X_{r'}$ such that $q_{r'+1}^*(f_{2n+2})=1$,
$f_{2n+2}\bigr|_L=0$, $p_D(v_{2n+2})<\epsilon_{2n+2}$ and
$(I+T_{2n+2})a=u$, where $T_{2n+2}x=T_{2n+1}x+f_{2n+2}(x)v_{2n+2}$.
The inclusion $a\in \Lambda$ means that $a=\alpha(m_{2n+2})$ with
$m_{2n+2}\in \N\setminus\{m_j:j\leq 2n+1\}$. Since
$u=\beta(s_{2n+2})$ and $f_{2n+1}\bigr|_L=0$, (\ref{ttt4}) with
$k=2n+1$ implies that (\ref{ttt4}) with $k=2n+2$ is also satisfied.
The equality $\tau(\alpha(m_{2n+2}))=\sigma(\beta(s_{2n+2}))=r'$
gives (\ref{ttt2}) for $k=2n+2$. The properties (\ref{ttt0}),
(\ref{ttt1}) and (\ref{ttt5}) for $k=2n+2$ are satisfied by
construction. It remains to notice that (\ref{ttt3}) for $k=2n+2$
follows from (\ref{ttt3}) for $k=2n$ and from the equalities
(\ref{TT2}) and (\ref{TT3}). This concludes the inductive
construction.

Now we consider the operator $T$ given by the formula
$$
Tx=\lim_{k\to\infty}T_kx=\sum_{j=1}^\infty f_j(x)v_j.
$$
Formula (\ref{TT0}) and Lemma~\ref{nuc} show that $T$ is a
well-defined continuous linear operator on $X$ such that $I+T$ is
invertible. Passing to the limit as $k\to\infty$ in (\ref{ttt4}), we
see that $(I+T)\alpha(m_j)=\beta(s_j)$ for every $j\in\N$. By
(\ref{ttt1}) and (\ref{ttt3}), $j\mapsto m_j$ and $j\mapsto s_j$ are
bijections of $\N$ onto itself. Taking into account that
$\alpha:\N\to A$ and $\beta:\N\to B$ are also bijections, we see
that $(I+T)(A)=B$. Now according (\ref{ttt2}) we have additionally
that $(I+T)(A_j)=B_j$ for every $j\in\Z_+$. Thus $R=I+T$ is an
isomorphism we were after.
\end{proof}

\begin{corollary}\label{GRgrgr} Let $X$ be a Fr\'echet space, whose
topology is defined by an increasing sequence $\{p_n\}_{n\in\N}$ of
seminorms. Let also $E$ and $F$ be two dense countably dimensional
subspaces of $X$ such that $\overline{E\cap \ker
p_n}=\overline{F\cap \ker p_n}$ for every $n\in\N$. Then there
exists $R\in GL(X)$ such that $R(E)=F$.
\end{corollary}

\begin{proof} For $n\in\N$, let $X_n=\overline{E\cap \ker
p_n}=\overline{F\cap \ker p_n}$. If there exists $n\in\N$ such that
$X_n$ is finite dimensional, then there is a continuous seminorm $q$
on $X$ such that $E\cap\ker q=F\cap \ker q=\{0\}$. Then $E,F\in{\cal
M}_1$. By Corollary~\ref{GR2} there is an isomorphism $S:E\to F$.
The unique continuous extension of $S$ to $R\in L(X)$ is an
isomorphism such that $R(E)=F$.

It remains to consider the case when each $X_n$ is infinite
dimensional. An easy application of Lemma~\ref{lile}, allows us for
every $n\in\N$ to choose dense countable $p_{n+1}$-independent
subsets $A_n$ and $B_n$ of $X_n$ such that $E\cap
X_n=\spann(A_n)\oplus (E\cap X_{n+1})$ and $F\cap
X_n=\spann(B_n)\oplus (F\cap X_{n+1})$. By Theorem~\ref{GRGRGR},
there is $R\in GL(X)$ such that $R(A_n)=B_n$ for every $n\in\N$.
Since $E$ is the linear span of the union of $A_n$ and $F$ is the
linear span of the union of $B_n$, we have $R(E)=F$.
\end{proof}

\subsection{Isomorphisms of spaces in ${\cal M}_2$}

\begin{lemma}\label{isia} Let $X\in {\cal F}_1$, $E,F$ be dense
countably dimensional subspaces of $X\times\omega$ and
$E_0=\{u\in\omega:(0,u)\in E\}$ and $F_0=\{u\in\omega:(0,u)\in F\}$.
Assume also that that at least one of the following conditions is
satisfied$:$
\begin{itemize}\itemsep=-3pt
\item[{\rm(a)}]$F_0$ and $E_0$ are both finite dimensional$;$
\item[{\rm(b)}]$\omega/\overline{F_0}$ and $\omega/\overline{E_0}$ are both finite
dimensional$;$
\item[{\rm(c)}]each of the $4$ spaces $F_0$, $E_0$,
$\omega/\overline{F_0}$ and $\omega/\overline{F_0}$ is infinite
dimensional.
\end{itemize}
Then there is $T\in GL(X\times\omega)$ such that $T(E)=F$.
\end{lemma}

\begin{proof} Since $X\in{\cal F}_1$, there is a continuous seminorm
$p$ on $X\times\omega$ such that $\ker p=\{0\}\times\omega$. If
$F_0$ and $E_0$ are both finite dimensional, we can find a
continuous norm $q$ on $X\times\omega$ such that $\ker q\subseteq
\{0\}\times\omega$ and $\ker q\cap \{0\}\times (E_0+F_0)=\{0\}$. It
follows that $\ker q\cap E=\{0\}$ and $\ker q\cap F=\{0\}$. Thus
$E,F\in{\cal M}_1$. By Corollary~\ref{GR2}, there is an isomorphism
$S:E\to F$. Then the unique continuous extension of $S$ to $T\in
L(X\times\omega)$ is an isomorphism and $T(E)=F$.

If (b) is satisfied then $Y=\overline{F_0}\cap \overline{E_0}$ is a
closed linear subspace of $\omega$ of finite codimension. Then
$\omega=Y\oplus L$, where $L$ is a finite dimensional subspace of
$\omega$. Thus $X\times\omega$ can be naturally identified with
$X_1\times Y$, where $X_1=X\times L$. Since $X\in{\cal F}_1$ and $L$
is finite dimensional, we have $X_1\in {\cal F}_1$. Furthermore, $Y$
is isomorphic to $\omega$ and $\{u\in Y:(0,u)\in E\}$ and $\{u\in
Y:(0,u)\in F\}$ are both dense in $Y$. Thus, without loss of
generality, we may assume that both $E_0$ and $F_0$ are dense in
$\omega$. Then by Theorem~\ref{omeg}, there is $S\in GL(\omega)$
such that $S(E_0)=F_0$. If (c) is satisfied, then by
Lemma~\ref{icod}, there is $S\in GL(\omega)$ such that $S(E_0)=F_0$.

By Lemma~\ref{lile} applied to a Hamel basis in $F$, there is a
dense in $X\times\omega$ $p$-independent subset of $F$. By Zorn's
lemma, there is a dense in $X\times\omega$ maximal $p$-independent
subset $A$ of $F$. Then $F$ is an algebraic direct sum of
$\spann(A)$ and $\{0\}\times F_0$:
$$
F=\spann(A)\oplus\{0\}\times F_0.
$$
Since $I_X\oplus S$ is an isomorphism of $X\times\omega$ onto
itself, $G=(I_X\oplus S)(E)$ is a dense countably dimensional
subspace of $X$. Furthermore, $G_0=\{u\in\omega:(0,u)\in G\}$
coincides with $S(E_0)=F_0$.

By Lemma~\ref{lile}, there is a dense in $X\times\omega$
$p$-independent subset of $G$. By Zorn's lemma, there is a dense in
$X\times\omega$ maximal $p$-independent subset $B$ of $G$. Then $G$
is an algebraic direct sum of $\spann(B)$ and $\{0\}\times
G_0=\{0\}\times F_0$:
$$
G=\spann(B)\oplus\{0\}\times F_0.
$$
Since both $A$ and $B$ are dense in $X\times\omega$ and
$p$-independent, Theorem~\ref{GR} furnishes us with $R\in
GL(X\times\omega)$ such that $R(B)=A$ and $Rx=x$ for every $x\in
\ker p=\{0\}\times\omega$. Thus using the last two displays, we
conclude that $R(G)=F$. Hence $T(E)=F$, where $T\in
GL(X\times\omega)$ is given by $T=R\circ (I_X\oplus S)$.
\end{proof}

\subsection{Action of $GL(X)$ on subspaces $E\in{\cal M}_0$ for $X\in{\cal
F}_3$}

\begin{lemma}\label{sram0} Let $X\in{\cal F}_2\cup{\cal F}_3$. Then $X$ has
subspaces isomorphic to $\omega$.
\end{lemma}

\begin{proof} By definition of the classes ${\cal F}_2$ and ${\cal
F}_3$, there is an increasing sequence $\{p_n\}_{n\in\N}$ of
seminorms defining the topology of $X$ such that $\ker p_n/\ker
p_{n+1}\neq\{0\}$ for every $n\in\N$. Hence, we can pick $x_n\in
\ker p_n\setminus\ker p_{n+1}$ for each $n\in\N$. Let $V$ be the
closed linear span of $x_n$ for $n\in\N$. It is easy to verify that
the series $\sum c_nx_n$ converges for every sequence $c=\{c_n\}$ of
complex numbers and that the map $c\mapsto \sum c_nx_n$ provides an
isomorphism between $\omega$ and $V$. Thus $X$ has subspaces
isomorphic to $\omega$.
\end{proof}

\begin{lemma}\label{sram} Let $X\in{\cal F}_3$. Then for every two
subspaces $Y,Z$ of $X$ isomorphic to $\omega$, there exists a
subspace $W$ of $X$ isomorphic to $\omega$ such that $Y\subset W$,
$Z\subset W$ and both $W/Y$ and $W/Z$ are infinite dimensional.
\end{lemma}

\begin{proof}
Let now $Y,Z$ be two subspaces of $X$ isomorphic to $\omega$. Then
$V=\overline{Y+Z}$ carries weak topology since its dense subspace
$Y+Z$ carries weak topology. Hence $V$ is isomorphic to $\omega$.
Since every isomorphic to $\omega$ subspace of a Fr\'echet space is
complemented, $X=V\oplus M$, where $M$ is a closed linear subspace
of $X$. It is straightforward to see that $M\in{\cal F}_3$ and
therefore by Lemma~\ref{sram0}, there is a closed linear subspace
$N$ of $M$ isomorphic to $\omega$. Now $W=V\oplus N$ has all desired
properties.
\end{proof}

\begin{lemma}\label{tram} Let $X\in {\cal F}_3$. Then $GL(X)$ acts
transitively on subspaces of $X$ isomorphic to $\omega$. Furthermore
$GL(X)$ acts transitively on countably dimensional subspaces of $E$
of $X$ such that $E\in {\cal M}_0$.
\end{lemma}

\begin{proof} Let $Y$ and $Z$  be two subspaces of $X$ isomorphic to
$\omega$ and $E$ and $F$ be dense countably dimensional subspaces of
$Y$ and $Z$ respectively. In order to prove the lemma, it suffices
to find $T\in GL(X)$ such that $T(E)=F$ (the equality $T(Y)=Z$
follows).

By Lemma~\ref{sram}, there is a subspace $W$ of $X$ isomorphic to
$\omega$ such that $Y\subset W$, $Z\subset W$ and both $W/Y$ and
$W/Z$ are infinite dimensional. Since every isomorphic to $\omega$
subspace of a Fr\'echet space is complemented, $X=W\oplus V$, where
$V$ is a closed linear subspace of $X$. By Lemma~\ref{icod}, there
is $S\in GL(W)$ such that $S(E)=F$. Clearly, $T=S\oplus {\rm
Id}_V\in GL(X)$ and $T(E)=F$.
\end{proof}

\subsection{Proof of Theorem~\ref{grgr11}}

Obviously subspaces of a topological space with weak topology also
carry weak topology. Furthermore the topology of $X$ is weak if $X$
possesses a dense subspace with weak topology. These observation
together with Remark~\ref{MJ} and Theorem~\ref{omeg} immediately
imply (\ref{grgr11}.1). For the sake of brevity we denote ${\cal
E}={\cal E}(X)$.

By definition, $X\notin{\cal F}_0$ if and only if $X$ possesses a
non-trivial continuous seminorm. Clearly, ${\cal E}\subseteq{\cal
M}_1$ if and only if $X$ possesses a continuous norm, that is $X\in
{\cal F}_1$. Finally, by Corollary~\ref{GR2}, every $E,F\in{\cal
E}\cap{\cal M}_1$ are isomorphic, which proves (\ref{grgr11}.2).

By Lemma~\ref{sram0}, every $X\in {\cal F}_2\cup{\cal F}_3$ has a
closed linear subspace $Y$ isomorphic to $\omega$. Since every
isomorphic to $\omega$ subspace of a Fr\'echet space is
complemented, $X=Y\oplus Z$, where $Z$ is a closed linear subspace
of $X$. Then $Z$ possesses a non-trivial continuous seminorm $p$
(otherwise $X\in {\cal F}_0$). Then the seminorm $q(y+z)=p(z)$ for
$y\in Y$ and $z\in Z$ is a non-trivial continuous seminorm on $X$.
By Corollary~\ref{nnn1}, there is a dense countably dimensional
subspace $E_1$ of $X$ such that $E_1\cap \ker q=\{0\}$. Let also
$E_2$ be a dense countably dimensional subspace of $Y$ and
$E=E_1+E_2$. Obviously, $E\in{\cal E}$. Since $E_2\subset Y$,
$E_1\cap \ker q=\{0\}$ and $Y\subseteq \ker q$, $E=E_1\oplus E_2$
(in algebraic sense) and $E_2=E\cap Y$. In particular, $E_2$ is
closed in $E$ and carries weak topology. Furthermore, $q$ provides
(in a natural way) a continuous norm on $E/E_2$. By Remark~\ref{MJ},
$E\in {\cal M}_2$. Thus ${\cal E}\cap{\cal M}_2\neq\varnothing$. On
the other hand, from (\ref{grgr11}.1) and (\ref{grgr11}.2) it
follows that ${\cal E}\cap{\cal M}_2=\varnothing$ if $X\in {\cal
F}_0\cup{\cal F}_1$. Thus ${\cal E}\cap{\cal M}_2\neq\varnothing$ if
and only if $X\in {\cal F}_2\cup{\cal F}_3$.

Now let $X\in{\cal F}_2$. Then  $X=Y\oplus Z$, where $Z\in{\cal
F}_1$ and $Y$ is isomorphic to $\omega$. For $E\in{\cal E}$, we
denote $E_0=E\cap Y$ and $E_1=\overline{E_0}$. Clearly ${\cal
E}\cap{\cal M}_2$  is the union of two disjoint subsets ${\cal E}_1$
and ${\cal E}_2$, where ${\cal E}_1=\{E\in {\cal E}\cap{\cal
M}_2:\dim Y/E_1<\infty\}$ and ${\cal E}_2=\{E\in {\cal E}\cap{\cal
M}_2:\dim Y/E_1=\infty\}$. By Lemma~\ref{isia}, every two spaces in
${\cal E}_1$ are isomorphic and every two spaces in ${\cal E}_2$ are
isomorphic. Let now $E\in {\cal E}_1$ and $F\in {\cal E}_2$. In
order to show that ${\cal E}\cap{\cal M}_2$ consists of exactly two
isomorphism classes, it remains to verify that $E$ and $F$ are
non-isomorphic. Assume the contrary. That is, there is an
isomorphism $T\in L(X)$ such that $T(E)=F$. Since $Z$ possesses a
continuous norm and $Y/E_1$ is finite dimensional, $X/E_1$ also
possesses a continuous norm. Since $T$ is an isomorphism, $X/T(E_1)$
possesses a continuous norm. Since $E_0$ is a subspace of $E$
carrying weak topology, $T(E_0)$ is a subspace of $F$ carrying weak
topology. It is straightforward to verify that every subspace of $F$
carrying weak topology is contained in $F_0+L$, where $L$ is a
finite dimensional subspace of $F$. Thus $T(E_0)\subseteq F_0+L$.
Then $T(E_1)=\overline{T(E_0)}\subseteq \overline{F_0}+L=F_1+L$.
Since $X/T(E_1)$ possesses a continuous norm, $X/F_1+L$ also
possesses a continuous norm. Since $L$ is finite dimensional,
$X/F_1$ possesses a continuous norm. The latter is impossible since
$Y/F_1$ is an infinite dimensional space carrying weak topology.
This contradiction proves that ${\cal E}\cap {\cal M}_2$ contains
exactly two isomorphism classes if $X\in{\cal F}_2$.

Let $X\in{\cal F}_3$. We have to show that every $E,F\in{\cal E}\cap
{\cal M}_2$ are isomorphic. Indeed, let $E,F\in{\cal E}\cap {\cal
M}_2$. By definition of ${\cal M}_2$, there are closed linear
subspaces $E_1$ and $F_1$ of $E$ and $F$ respectively such that both
$E_1$ and $F_1$ carry weak topology and both $E/E_1$ and $F/F_1$ are
infinite dimensional and possess a continuous norm. By
Lemma~\ref{tram}, there is an isomorphism $T\in L(X)$ such that
$T(E_1)=F_1$. Then $F_1$ is a common closed subspace of $T(E)$ and
of $F$ and both $T(E)/F_1$ and $F/F_1$ possess a continuous norm.
Hence there is a continuous seminorm $p$ on $X$ such that
$F_1\subset\ker p$ and $T(E)\cap\ker p=F\cap\ker p=F_1$. Using
Lemma~\ref{lile}, we can find dense in $X$ $p$-independent sets $A$
and $B$ such that $\spann(B)\oplus F_1=T(E)$ and $\spann(A)\oplus
F_1=F$ (in algebraic sense). By Theorem~\ref{GR}, there is an
isomorphism $S\in L(X)$ such that $S(B)=A$ and $Sx=x$ for $x\in\ker
p$. Since $F_1\subset\ker p$, it follows that $S(F_1)=F_1$. Hence
$R(E_1)=F_1$, where $R=ST$. Next, the relation $\spann(B)\oplus
F_1=T(E)$ implies $\spann(T^{-1}(B))\oplus E_1=E$. Furthermore,
$R(T^{-1}(B))=S(B)=A$. It immediately follows that $R(E)=F$. Thus
$E$ and $F$ are isomorphic, which completes the proof of
(\ref{grgr11}.3).

It remains to verify (\ref{grgr11}.4). The fact that $X\in{\cal
F}_3$ if and only if ${\cal E}\cap{\cal M}_3\neq\varnothing$ is
obvious. We have to show that there are uncountably many pairwise
non-isomorphic spaces in ${\cal E}\cap {\cal M}_3$. Let $E_n\in
{\cal E}\cap {\cal M}_3$ for $n\in\N$. It suffices to find a $E\in
{\cal E}\cap {\cal M}_3$ non-isomorphic to each $E_n$. Since
$X\in{\cal F}_3$, we can choose an increasing sequence
$\{p_n\}_{n\in\N}$ of continuous seminorms on $X$ defining the
topology of $X$ such that $X_n/X_{n+1}$ is infinite dimensional for
every $n\in\Z_+$, where $X_0=X$ and $X_n=\ker p_n$ for $n\in\N$.
Since $E_k\in{\cal M}_3$, $\overline{E_k\cap X_n}\in{\cal F}_3$ for
every $n\in\Z_+$ and $k\in\N$. By Lemma~\ref{fufu}, there is a
sequence $r\in\Omega$ (see Section~\ref{FFFUUU}) such that
$\overline{E_k\cap X_n}\notin{\cal F}_{[r]}$ for every $n\in\Z_+$
and $k\in\N$. By Lemma~\ref{cuddly}, there is a closed subspace $Y$
of $X_1$ such that $Y\cap X_n$ is infinite dimensional for every
$n\in\N$ and $Y\in{\cal F}_{[r]}$. Now for each $m\in\N$ we can
choose a dense countably dimensional subspace $F_m$ of $Y\cap X_m$.
By Corollary~\ref{nnn1}, there is a dense in $X$ countably
dimensional subspace $F_0$ such that $F_0\cap X_1=\{0\}$. Let $E$ be
the linear span of the union of $F_m$ for $m\in\Z_+$. Clearly
$E\in{\cal E}$ and $E\cap X_n$ is infinite dimensional for every
$n\in\N$. It follows that $E\in {\cal M}_3$. It remains to show that
$E$ is non-isomorphic to each of $E_n$. Assume the contrary. That
is, there are $n\in\N$ and an isomorphism $T\in L(X)$ such that
$T(E_n)=E$. Continuity of $T$ implies that there are $c>0$ and
$k\in\N$ such that $p_1(Tx)\leq cp_k(x)$ for every $x\in X$. In
particular, $T(X_k)\subseteq X_1$. Hence $T(X_k\cap E_n)\subseteq
X_1\cap E$ and therefore $T(\overline{X_k\cap E_n})\subseteq
\overline{X_1\cap E}=Y$. That is, $\overline{X_k\cap E_n}$ is
isomorphic to a subspace of $Y$. Since $Y\in{\cal F}_{[r]}$, it
follows that $\overline{X_k\cap E_n}\in{\cal F}_{[r]}$. This
contradiction completes the proof of (\ref{grgr11}.4) and that of
Theorem~\ref{grgr11}.

\section{Orbitality for countably dimensional $E\in{\cal M}$}

\subsection{Strong orbitality for $E\in{\cal M}_0\cup{\cal M}_1$}

\begin{proposition}\label{boho} Every countably
dimensional $E\in {\cal M}_0$ is strongly orbital. In particular,
$E$ does not have the invariant subset property.
\end{proposition}

\begin{proof} It is well-known (see, for instance, \cite{bope})
that $\omega$ supports a hypercyclic operator. Thus, we can take
$S\in L(\omega)$ and $x\in \omega$ such that $x$ is a hypercyclic
vector for $S$. Now let $F=\spann\{S^nx:n\in\Z_+\}$. Then $F$ is a
dense countably dimensional subspace of $\omega$. Obviously,
$F\in{\cal M}_0$ and $F$ is strongly orbital. By
Corollary~\ref{isisis},  $E$ is isomorphic to $F$ and therefore $E$
is strongly orbital as well.
\end{proof}

The following result is an immediate corollary of Lemma~5.2 in
\cite{ss}.

\begin{lemma}\label{fac} Let $p$ be a non-trivial continuous
seminorm on a separable Fr\'echet space $X$. Then there exists a
hypercyclic $T\in L(X)$ such that $Tx=x$ for every $x\in \ker p$.
\end{lemma}

\begin{proposition}\label{boho1} Every countably
dimensional $E\in {\cal M}_1$ is strongly orbital. In particular,
$E$ does not have the invariant subset property.
\end{proposition}

\begin{proof} Pick a dense Hamel basis $A$ in $E$.
Let $X$ be the completion of $E$. Then $X\in{\cal F}$. Since
$E\in{\cal M}_1$, there is a continuous seminorm $p$ on $X$ such
that $E\cap \ker p=\{0\}$. It follows that $A$ is $p$-independent.
By Lemma~\ref{fac}, there is a hypercyclic operator $R\in L(X)$ such
that $Rx=x$ for each $x\in\ker p$. Let $u\in X$ be a hypercyclic
vector for $R$. Let $F=\spann(O)$, where
$O=O(T,u)=\{T^nu:n\in\Z_+\}$. Observe that $O$ is $p$-independent.
Indeed, otherwise $R^n u\in L+\ker p$, where
$L=\spann\{u,Ru,\dots,R^{n-1}u\}$ for some $n\in\N$. Since $Rx=x$
for $x\in\ker p$, it follows that $O\subset L+\ker p$. Since $L$ is
finite dimensional and $\ker p$ is a closed linear subspace of $X$
of infinite codimension, $L+\ker p$ is a proper closed subspace of
$X$. Then the inclusion $O\subset L+\ker p$ contradicts density of
$O$ in $X$. Thus $O$ is $p$-independent and therefore $F\cap\ker
p=\{0\}$. Hence $F\in{\cal M}_1$. Obviously, $F$ is strongly
orbital. By Corollary~\ref{GR2}, $E$ and $F$ are isomorphic. Hence
$E$ is strongly orbital.
\end{proof}

\subsection{Strong orbitality for countably dimensional $E\in{\cal M}_3$}

\begin{lemma}\label{hyhoho} Let $X$ be a separable Fr\'echet space,
whose topology is defined by an increasing sequence
$\{p_n\}_{n\in\N}$ of seminorms. Let also $\{X_n\}_{n\in\Z_+}$ be a
sequence of closed linear subspaces of $X$ such that $X_0=X$,
$X_{n+1}\subseteq X_{n}\cap\ker p_{n+1}$ and $X_n/(X_n\cap \ker
p_{n+1})$ is infinite dimensional for every $n\in\Z_+$. Then there
are $T\in L(X)$ and $x\in X$ such that $O(T,x)\cap \ker p_n$ is a
dense subset of $X_n$ for every $n\in\Z_+$.
\end{lemma}

\begin{proof} By Lemma~\ref{BuBuBu}, there exist
$\{u_{n,k}:n,k\in\Z_+\}\subset X$ and $\{f_{n,k}:n,k\in\Z_+\}\subset
X'$ such that (\ref{BuBuBu1}-\ref{BuBuBu4}) are satisfied. By
Corollary~\ref{l22a}, there is a Banach disk $D$ in $X$ such that
$u_{n,k}\in D$ for every $n,k\in\Z_+$. For each $n\in\N$, consider
$T_n\in L(X)$ defined by the formula
$$
T_{n}x=\sum_{k=0}^\infty \frac{2^{-k}f_{n,k}(x)}
{p_{n+1}^*(f_{n,k})p_D(u_{n-1,k})}u_{n-1,k}.
$$
Since $|f_{n,k}(x)|\leq p_{n+1}(x)p_{n+1}^*(f_{n,k}(x))$, the series
in the above display converges absolutely in the Banach space $X_D$
and $p_D(T_nx)\leq 2p_{n+1}(x)$ for every $x\in X$. Hence $T_n:X\to
X_D$ is a continuous linear operator and therefore $T_n\in L(X)$.
Furthermore, (\ref{BuBuBu1}) ensures that
\begin{equation}
\text{$T_n(X)\subseteq X_{n-1}$ for every $n\in\N$.}\label{trat1}
\end{equation}
Since $X_n\subseteq\ker p_n$ for $n\in\N$, (\ref{trat1}) guarantees
that the series $\sum T_n$ converges pointwise. By the uniform
boundedness principle, the formula
$$
Tx=\sum_{n=1}^\infty T_nx=\sum_{n=1}^\infty\sum_{k=0}^\infty
\frac{2^{-k}f_{n,k}(x)} {p_{n+1}^*(f_{n,k})p_D(u_{n-1,k})}u_{n-1,k}
$$
defines a continuous linear operator $T$ on $X$. For $n\in\Z_+$, let
$E_n=\spann\{u_{n,k}:k\in\Z_+\}$ and $F_n=\spann\{u_{m,k}:m\geq n,\
k\in\Z_+\}$. Clearly $F_n$ is the sum of $E_m$ for $m\geq n$. By
(\ref{BuBuBu1}), $F_n$ is a dense countably dimensional subspace of
$X_n$. According to (\ref{BuBuBu4}),
\begin{equation}
\text{$T$ vanishes on $E_0$ and $S_n=T\bigr|_{E_n}:E_n\to E_{n-1}$
is an invertible linear operator for each $n\in\N$.}\label{brr}
\end{equation}

Now pick a countable set $A\subset F_0$ such that $A\cap F_k$ is
dense in $F_k$ for every $k\in\Z_+$. By (\ref{BuBuBu1}), each $F_k$
is dense in $X_k$ and therefore $A\cap F_k$ is dense in $X_k$. For
each $a\in A$, let $k(a)=\max\{m\in\Z_+:a\in F_m\}$. Then $a\in
F_{k(a)}$ and $A_m=\{a:k(a)=m\}$ is dense in $X_m$ for every
$m\in\Z_+$. By definition of $F_m$, for every $a\in A$, there is
$r(a)\geq k(a)$ such that $a\in
E_{k(a)}+E_{k(a)+1}+{\dots}+E_{r(a)}$. In other words,
$$
a=\sum_{j=k(a)}^{r(a)} w(j,a),\ \ \text{where}\ w(j,a)\in E_j.
$$
Now by (\ref{brr}),
\begin{equation}\label{brr1}
\text{for $u\in E_j$, $T^nu\in E_{j-n}$ if $n\leq j$ and $T^nu=0$ if
$n>j$.}
\end{equation}
For $u\in E_j$ and $n\in\Z_+$, we denote $u^{[n]}=u$ if $n=0$ and
$u^{[n]}=S_{j+n}^{-1}\dots S_{j+1}^{-1}u$ if $n>0$, where $S_l$ are
defined in (\ref{brr}). According to (\ref{brr}),
\begin{equation}\label{brr2}
\text{$u^{[n]}\in E_{j+n}$ and $T^mu^{[n]}=u^{[n-m]}$ if $m\leq n$.}
\end{equation}
For $a\in A$ and $n\in\Z_+$, we denote
$$
a^{[n]}=\sum_{j=k(a)}^{r(a)}u(j,a)^{[n]}.
$$
Now we enumerate $A$: $A=\{a_l:l\in\N\}$. Define the sequence
$\{n_j\}_{j\in\N}$ of non-negative integers by the formula $n_1=1$
and $n_{j+1}=n_j+r(a_{j+1})+j$ for $j\in\N$. Consider the series
$$
x=\sum_{l=1}^\infty a_l^{[n_l]}=\sum_{l=1}^\infty
\sum_{j=k(a_l)}^{r(a_l)} u(j,a_l)^{[n_l]}.
$$
First, we show that this series converges in $X$ and therefore $x$
is well-defined. Indeed, let $s\in\N$. By (\ref{brr2}),
$u(j,a_l)^{[n_l]}\in E_{j+n_l}\subset F_{j+n_l}\subseteq
F_{n_l}\subseteq F_l\subset X_l$. Hence $p_s(a_l^{[n_l]})=0$ for
$l\geq s$ and therefore the series in the above display converges
absolutely. Now let $m\in\N$, $l<m$ and $k(a)\leq j\leq r(a)$. Since
$n_m-n_l>r(l)$, (\ref{brr1}) implies that
$T^{n_m}(u(j,a_l)^{[n_l]})=0$. Hence  $T^{n_m}(a_l^{[n_l]})=0$.
Then, using (\ref{brr1}), we have
$T^{n_m}(a_l^{n_l})=a_l^{[n_l-n_m]}$ for $l>m$. Hence
$$
T^{n_m}x=a_m+\sum_{l=m+1}^\infty
a_l^{[n_l-n_m]}=a_m+\sum_{l=m+1}^\infty \sum_{j=k(a_l)}^{r(a_l)}
u(j,a_l)^{[n_l-n_m]}.
$$
Since for $l>m$, $n_l-n_m>r(a_l)+l>r(a_l)+m$, (\ref{brr2}) and the
above display imply that
$$
a_m-T^{n_m}x\in F_m\subseteq X_m.
$$
Hence $p_s(a_m-T^{n_m}x)=0$ for $m\geq s$ and therefore
$a_m-T^{n_m}x\to 0$ in $X$. Moreover, by (\ref{BuBuBu1}) and
(\ref{BuBuBu2}), for every $j\in\Z_+$, there is $l\in\Z_+$ such that
$T^jx\in X_l\setminus\ker p_{l+1}$. It follows that $O(T,x)\cap \ker
p_n\subseteq X_n$ for every $n\in\N$. Furthermore, $T^{n_m}x\in
X_{k(a_m)}\setminus \ker p_{k(a_m)+1}$ for each $m\in\N$. Since
$A_s=\{a\in A:k(a)=s\}$ is dense in $X_s$ for every $s\in\Z_+$, the
relation $a_m-T^{n_m}x\to 0$ ensures that $\{T^{n_m}x:k(a_m)=s\}$ is
a dense subset of $X_s$ for each $s\in\Z_+$. That is, $O(T,x)\cap
\ker p_{n}$ is a dense subset of $X_n$ for every $n\in\Z_+$.
\end{proof}

\begin{proposition}\label{nonom} Let $E\in{\cal M}_3$ be countably
dimensional. Then $E$ is strongly orbital and therefore does not
have the invariant subset property.
\end{proposition}

\begin{proof} Let $X$ be the completion of $E$. Since $E\in{\cal
M}_3$, we can pick an increasing sequence $\{p_n\}_{n\in\N}$ of
continuous seminorms on $X$ defining the topology of $X$ such that
$E\cap \ker p_1$ has infinite codimension in $E$ and $E\cap \ker
p_{n+1}$ has infinite codimension in $E\cap \ker p_n$ for each
$n\in\N$. Let $X_0=X$ and $X_n=\overline{E\cap \ker p_n}$ for
$n\in\N$. It is easy to verify that $X_{n+1}\subseteq X_{n}\cap\ker
p_{n+1}$ and $X_n/(X_n\cap \ker p_{n+1})$ is infinite dimensional
for every $n\in\Z_+$. By Lemma~\ref{hyhoho}, there are $T\in L(X)$
and $x\in X$ such that $O(T,x)\cap \ker p_n$ is a dense subset of
$X_n$ for every $n\in\Z_+$. Clearly $F=\spann(O(T,x))$ is strongly
orbital. On the other hand, $X_n=\overline{E\cap \ker
p_n}=\overline{F\cap \ker p_n}$ for each $n\in\N$. By
Corollary~\ref{GRgrgr}, $E$ and $F$ are isomorphic. Hence $E$ is
also strongly orbital.
\end{proof}

\subsection{Invariant subset property for countably dimensional $E\in{\cal M}_2$}

\begin{lemma}\label{oper1} Let $X\in{\cal F}$ be the direct sum $X=Y\oplus Z$
of its closed linear subspaces such that $Y$ is isomorphic to
$\omega$ and $Z$ is non-isomorphic to $\omega$. Then there is a
dense countably dimensional subspace $E$ of $X$ such that $E$ does
not have the invariant subset property and $E\cap Y$ is dense in
$Y$.
\end{lemma}

\begin{proof} Since $Z$ is non-isomorphic to $\omega$, there is a
non-trivial continuous seminorm $p$ on $X$ such that $Y\subseteq\ker
p$. Since $Y$ is isomorphic to $\K^\N$, we can pick
$\{x_n:n\in\N\}\subset Y$ and $\{f_n:n\in\N\}\subset X'$ such that
$f_n\bigr|_Z=0$  for each $n\in\N$, $f_n(x_m)=\delta_{n,m}$ for
every $m,n\in\N$ and the map $c=\{c_n\}\mapsto
\sum\limits_{k=1}^\infty c_nx_n$ is an isomorphism of $\K^\N$ and
$Y$.

Pick any dense countably dimensional subspace $F$ of $Z$. By
Corollary~\ref{l22a}, there exists a Banach disk $D$ in $Z$ such
that $F\subset Z_D=X_D$. By Lemma~\ref{bububu}, there exists a Hamel
basis $\{x_{-n}\}_{n\in\Z_+}$ in $F$ and a sequence
$\{f_{-n}\}_{n\in\N}$ in $X'_p$ such that
$f_{-n}(x_{-m})=\delta_{n,m}$ for every $m,n\in\Z_+$. Since
$f_{-n}\in X'_p$, $f_{-n}$ vanishes on $Y$ for each $n\in\Z_+$ and
therefore $f_{-n}(x_k)=0$ if $k\in\N$ and $n\in\Z_+$. Since $f_n$
vanishes on $Z$ for $n\in\N$, $f_{k}(x_{-n})=0$ if $k\in\N$ and
$n\in\Z_+$. Thus $f_n(x_m)=\delta_{n,m}$ for $m,n\in\Z$. Consider
$T\in L(X)$ defined by the formula
$$
Tx=\sum_{n=-\infty}^\infty c_nf_{n+1}(x)x_n,
$$
where $c_n=1$ for $n\geq 1$ and
$c_n=\frac{2^n}{p^*(f_{n+1})p_D(x_n)}$ for $n<0$. Obviously
$T\bigr|_Y$ is continuous and well defined. From the definition of
$c_n$ it follows that for $x\in Z$, the series in the above display
converges absolutely in the Banach space $X_D$ and therefore in $X$.
Thus $T$ is a well-defined continuous linear operator. Let
$G=\spann\{x_n:n\in\Z\}$. It is easy to see that $G$ is dense in $X$
and that $T\bigr|_G:G\to G$ is a bijective linear map. Thus we can
consider its inverse $S=(T\bigr|_G)^{-1}:G\to G$. It is
straightforward to verify that $S\bigr|_{Y\cap G}$  is the forward
shift:
$$
Sx=\sum_{n=1}^\infty f_n(x)x_{n+1}\ \ \text{for $x\in Y\cap G$}.
$$
The operator $S$ on the space $G$ is, of course, a bilateral
weighted (forward) shift with respect to the basis
$\{x_n\}_{n\in\Z}$.

First, we shall show that there is $w\in Y$ such that $f_1(w)\neq
0$, and $w$ is a hypercyclic vector for $T$. Indeed, let $w_0=x_1$
and choose a dense in $X$ countable subset $\{w_n:n\in\N\}$ of $G$.
We start with a few easy observations. First, from the definition of
$T$ it follows that $T^nw\in Z$ for all sufficiently large $n$ and
$p_D(T^nw)\to 0$ for every $w\in G$. Second, for each $w\in G$ and
$k\in\N$, $S^nw\in Y$ and $f_1(S^nw)={\dots}=f_k(S^nw)=0$ all for
sufficiently large $n$. Using this observations and the fact that
$T^nS^mw=T^{n-m}w$ and $T^mS^nw=S^{n-m}w$ for $w\in G$ and
$n,m\in\Z_+$, $n>m$, it is a standard exercise to see that if the
sequence $\{k_n\}_{n\in\Z_+}$ of non-negative integers with $k_0=0$
grows fast enough, then
\begin{itemize}\itemsep=-2pt
\item[(a)]
the sets $\{k\in\N:f_k(S^{k_n}w_n)\neq 0\}$ for $n\in\Z_+$ are
pairwise disjoint and therefore the series $\sum\limits_{n=0}^\infty
S^{k_n}w_n$ converges in $Y\subset X$ to $w\in Y$;
\item[(b)]
$T^{k_n}w-w_n\to 0$ in $X$.
\end{itemize}

Thus $w$ constructed in this way is a hypercyclic vector for $T$.
Since $f_1(w_0)=f_1(S^{k_0}w_0)=1$, (a) implies that
$f_1(S^{k_n}w_n)=0$ for $n\in\N$ and therefore $f_1(w)=1$.

Now we set $u_n=T^nw$ for $n\in\Z_+$ and
$u_n=\sum\limits_{j=1}^\infty f_j(w)x_{j-n}$ for $n\in\Z$, $n<0$.
From the definition of $T$ it follows that
\begin{equation}\label{shish}
\text{$T^mu_n=u_{n+m}$ for every $n\in\Z$ and $m\in\Z_+$}.
\end{equation}
Now let $E=\spann\{u_n:n\in\Z\}$. Since $w$ is a hypercyclic vector
for $T$, $E$ is dense in $X$. Since $u_n\in Y$ for $n\leq 0$,
$f_j(u_{j-1})=1$ for $j\in\N$ and $f_j(u_k)=0$ for $1\leq j\leq k$,
we see that $\spann\{u_n:n\leq 0\}$ is a dense subspace of $Y$.
Hence $E\cap Y$ is dense in $Y$. According to (\ref{shish}),
$T(E)\subseteq E$. It remains to show that $E$ does not have the
invariant subset property. In order to show that, it suffices to
demonstrate that every $v\in E\setminus\{0\}$ is hypercyclic vector
for $T$. Let $v\in E\setminus\{0\}$. Using (\ref{shish}) and the
equality $u_0=w$, we see that there is a non-zero polynomials $q$
and $n\in\N$ such that $T^nv=q(T)w$. According to Bourdon
\cite{bourd}, $q(T)w$ is a hypercyclic vector for $T$. Hence $T^nv$
is hypercyclic for $T$ and therefore $v$  is hypercyclic for $T$.
\end{proof}

\begin{proposition}\label{titi} Let $E\in{\cal M}_2$ be countably
dimensional. Then $E$ does not have the invariant subset property.
\end{proposition}

\begin{proof} Let $X$ be the completion of $E$. Since $E\in{\cal
M}_2$, there is a non-trivial continuous seminorm $p$ on $X$ such
that $\ker p\cap E$ is infinite dimensional and carries weak
topology. Let $Y$ be the closure in $X$ of $\ker p\cap E$. Then $Y$
is isomorphic to $\omega$ and therefore is complemented in $X$.
Hence there is a closed linear subspace $Z$ of $X$ such that
$X=Z\oplus Y$. Since $p$ is a non-trivial seminorm on $X$ vanishing
on $Y$, $Z$ is infinite dimensional. By Lemma~\ref{oper1}, there is
a dense countably dimensional subspace $F$ of $X$ such that $F\cap
Y$ is dense in $Y$ and $F$ does not have the invariant subset
property.

By Theorem~\ref{omeg}, there is $S\in GL(Y)$ such that $S(E\cap
Y)=F\cap Y$. Using Lemma~\ref{lile}, we can find dense in $X$
$p$-independent countable subsets $A$ and $B$ of $E$ and $F$
respectively such that $E=(E\cap Y)\oplus\spann(A)$ and $F=(F\cap
Y)\oplus\spann(B)$ (in the algebraic sense). By Theorem~\ref{GR},
there is $R\in GL(X)$ such that $R(A)=B$ and $Rx=x$ for $x\in \ker
p\supseteq Y$. Let $T=R\circ (S\oplus {\rm Id}_Z)$. Then the
properties of $A$, $B$, $R$ and $S$ listed above imply $T(E)=F$.
Thus $E$ and $F$ are isomorphic and therefore $E$ also does not have
the invariant subset property.
\end{proof}

\subsection{Non-orbitality for $X\in{\cal M}_2$}

\begin{proposition}\label{noorb}
Let $E\in{\cal M}_2$ be countably dimensional. Then $E$ is
non-orbital.
\end{proposition}

In order to prove the above proposition, we need the following
curious elementary lemma.

\begin{lemma}\label{idea} Let $J$ be a linear subspace of the space
$\K[z]$ of polynomials such that $(J+zJ)/J$ is finite dimensional.
Then either $J$ is finite dimensional or $J$ has finite codimension
in $\K[z]$.
\end{lemma}

\begin{proof} Assume that $J$ is infinite dimensional. We have to
show that the codimension of $J$ in $\K[z]$ is finite. Since
$(J+zJ)/J$ is finite dimensional, there exist $u_1,\dots,u_k\in
\K[z]$ such that $zJ\subseteq J+L$, where
$L=\spann\{u_1,\dots,u_k\}$. Since $J$ is infinite dimensional,
there is $p\in J$ such that $n=\deg p>\deg u_j$ for every
$j\in\{1,\dots,k\}$. Next, if $m\in\N$, $m\geq n$ and there is $q\in
J$ satisfying $\deg q=m$, then the inclusion $zJ\subseteq J+L$
implies that there is $r\in L$ for which $zq-r\in J$. Since $\deg
r\leq \deg q$, $\deg (zq-r)=m+1$. Thus $J$ contains a polynomial of
degree $m+1$ whenever $m\geq n$ and $J$ contains a polynomial of
degree $m$. Since $J$ contains a polynomial of degree $n$, it
follows that $J$ contains a polynomial of degree $m$ for every
$m\geq n$. Hence $\dim(\K[z]/J)\leq n$.
\end{proof}

\begin{proof}[Proof of Proposition~\ref{noorb}] Since $E\in{\cal
M}_2$, there exists a non-trivial continuous seminorm $p$ on $E$
such that $F=\ker p$ carries weak topology. Let $T\in L(E)$ and
$x\in E$. We have to show that $\spann(O(T,x))\neq E$. Assume the
contrary. That is, $\spann(O(T,x))=E$. Since $E$ is infinite
dimensional, the vectors $T^nx$ are linearly independent. Hence the
map $\Phi:\K[z]\to E$, $\Phi(p)=p(T)x$ is a linear bijection.
Consider the map $S:F\to E/F$, $Sx=Tx+F$. Then $S$ is a continuous
linear operator from the space $F$ carrying weak topology to the
space $E/F$ possessing a continuous norm. Since the continuous
linear image of a space with weak topology carries weak topology and
only finite dimensional spaces carry weak topology and a continuous
norm simultaneously, $S(F)$ is finite dimensional. That is, there is
a finite dimensional subspace $L$ of $E$ such that $T(F)\subseteq
F+L$. This means that $(J+zJ)/J$ is finite dimensional, where
$J=\Phi^{-1}(F)=\{p\in\K[z]:p(T)x\in F\}$. Since $F$ is infinite
dimensional and $\Phi:\K[z]\to E$ is a linear bijection, $J$ is
infinite dimensional. By Lemma~\ref{idea}, $J$ has finite
codimension in $\K[z]$. Since $\Phi(J)=F$ and $\Phi:\K[z]\to E$ is a
linear bijection, $F=\ker p$ has finite codimension in $E$, which
contradicts the non-triviality of $p$.
\end{proof}

\subsection{Proof of Theorems~\ref{th1} and~\ref{th1a}}

Let $E\in{\cal M}$ be countably dimensional. If $E\notin{\cal M}_2$,
then $E$ is strongly orbital according to Propositions~\ref{boho},
\ref{boho1} and~\ref{nonom}. By Lemma~\ref{orbi}, $E$ does not have
the invariant subset property if $E\notin{\cal M}_2$. If $E\in{\cal
M}_2$, then $E$ does not have the invariant subset property
according to Proposition~\ref{titi}. Finally, By
Proposition~\ref{noorb}, $E$ is non-orbital if $E\in{\cal M}_2$.
This proves Theorems~\ref{th1} and~\ref{th1a}.

\section{Proof of Theorem~\ref{main}}

We shall construct an operator $T$ with no non-trivial invariant
subspaces by lifting a non-linear map on a topological space to a
linear map on an appropriate topological vector space.

\subsection{A class of complete countably dimensional spaces}

Recall that a topological space $X$ is called {\it completely
regular} (or {\it Tychonoff}) if for every $x\in X$ and a closed
subset $F\subset X$ satisfying $x\notin F$, there is a continuous
$f:X\to \R$ such that $f(x)=1$ and $f\bigr|_F=0$. Equivalently, a
topological space is completely regular, if its topology can be
defined by a family of pseudometrics. Note that any subspace of a
completely regular space is completely regular and that every
topological group is completely regular.

Our construction is based upon the concept of the {\it free locally
convex space} \cite{usp}. Let $X$ be a completely regular
topological space. We say that a topological vector space $L_X$ is a
free locally convex space of $X$ if $L_X$ is locally convex,
contains $X$ as a subset with the topology induced from $L_X$ to $X$
being the original topology of $X$ and for every continuous map $f$
from $X$ to a locally convex space $Y$ there is a unique continuous
linear operator $T:L_X\to Y$ such  that $T\bigr|_X=f$. It turns out
that for every completely regular topological space $X$, there is a
free locally convex space $L_X$ unique up to an isomorphism leaving
points of $X$ invariant. Thus we can speak of {\it the} free locally
convex space $L_X$ of $X$. Note that $X$ is always a Hamel basis in
$L_X$. That is, as a vector space, $L_X$ consists of formal finite
linear combinations of elements of $X$. Identifying $x\in X$ with
the point mass measure $\delta_x$ on $X$ ($\delta_x(A)=1$ if $x\in
A$ and $\delta_x(A)=0$ if $x\notin A$), we can also think of
elements of $L_X$ as measures with finite support on the
$\sigma$-algebra of all subsets of $X$. Under this interpretation
$$
L_X^0=\{\mu\in L_X:\mu(X)=0\}
$$
is a closed hyperplane in the locally convex space $L_X$. If $f:X\to
X$ is a continuous map, from the definition of the free locally
convex space it follows that $f$ extends uniquely to a continuous
linear operator $T_f\in L(L_X)$. It is also clear that $L_X^0$ is
invariant for $T_f$. Thus the restriction $S_f$ of $T_f$ to $L_X^0$
belongs to $L(L_X^0)$.

According to Uspenskii \cite{usp}, $L_X$ is complete if and only if
$X$ is Dieudonne complete \cite{eng} and every compact subset of $X$
is finite. Since Dieudonne completeness follows from paracompactness
\cite{eng}, every regular countable topological space is Dieudonne
complete. Since every countable compact topological space is
metrizable \cite{eng}, for a countable $X$, finiteness of compact
subsets is equivalent to the absence of non-trivial convergent
sequences (a convergent sequence is trivial if it is eventually
stabilizing). Note also that a regular countable topological space
is automatically completely regular \cite{eng} and therefore we can
safely replace the term {\it completely regular} by {\it regular} in
the context of countable spaces. Thus we can formulate the following
corollary of the Uspenskii theorem.

\begin{proposition}\label{prop3} Let $X$ be a regular countable
topological space. Then the countably dimensional locally convex
topological vector spaces $L_X$ and $L_X^0$ are complete if and only
if there are no non-trivial convergent sequences in $X$.
\end{proposition}

The above proposition provides plenty of complete locally convex
spaces of countable algebraic dimension. We also need the shape of
the dual space of $L_X$. As shown in \cite{usp}, $L_X'$ can be
identified with the space $C(X)$ of continuous scalar valued
functions on $X$ in the following way. Every $f\in C(X)$ produces a
continuous linear functional on $L_X$ in the usual way:
$$
\langle f,\mu\rangle=\int f\,d\mu=\sum c_jf(x_j),\ \ \text{where}\ \
\mu=\sum c_j\delta_{x_j}
$$
and there are no other continuous linear functionals on $L_X$. Note
also that $L_X^0$ is the kernel of the functional $\langle {\bf
1},\cdot\rangle$, where ${\bf 1}$ is the constant 1 function.

\subsection{Operators $S_f$ with no invariant subspaces}

The following lemma is the main tool in the proof of
Theorem~\ref{main}.

\begin{lemma}\label{m1}Let $\tau$ be a regular topology on $\Z$ such that
$f:\Z\to \Z$, $f(n)=n+1$ is a homeomorphism of $\Z_\tau=(\Z,\tau)$
onto itself, $\Z_+$ is dense in $\Z_\tau$ and for every
$z\in\C\setminus\{0,1\}$, $n\mapsto z^n$ is non-continuous as a map
from $\Z_\tau$ to $\C$. Then the operators $T_f$ and $S_f$ are
invertible continuous linear operators on $L_{\Z_\tau}$ and
$L^0_{\Z_\tau}$ respectively and $S_f$ has no non-trivial closed
invariant subspaces.
\end{lemma}

\begin{proof} We shall prove the lemma in the case $\K=\C$. The case
$\K=\R$ follows easily by passing to complexifications.

We already know that $T_f$ and $S_f$ are continuous linear
operators. It is easy to see that $T_f^{-1}=T_{f^{-1}}$ and
$S_f^{-1}=S_{f^{-1}}$. Since $f^{-1}$ is also continuous, $T_f$ and
$S_f$ have continuous inverses.

Now let $\mu\in L^0_{\Z_\tau}\setminus\{0\}$. It remains to show
that $\mu$ is a cyclic vector for $S_f$. Assume the contrary. Then
there is a non-constant $g\in C(\Z_\tau)$ such that $\langle
S_f^n\mu,g\rangle=\langle T_f^n\mu,g\rangle=0$ for every $n\in\Z_+$.
Decomposing $\mu$ as a linear combination of point mass measures, we
have $\mu=\sum\limits_{k=-l}^l c_k\delta_k$ with $c_k\in\C$. Then
$\mu=\sum\limits_{j=0}^{2l}c_{j-k}T_f^{j}\delta_{-l}=p(T_f)\delta_{-l}$,
where $p$ is a non-zero polynomial. Then $0=\langle
T_f^n\mu,g\rangle=\langle T_f^n\delta_{-l},p(T_f)'g\rangle$ for
$n\in\Z_+$. Thus the functional $p(T_f)'g$ vanishes on the linear
span of $T_f^n\delta_{-l}$ with $n\in\Z_+$, which contains the
linear span of $\Z_+$ in $L_{\Z_\tau}$. Since $\Z_+$ is dense in
$\Z_\tau$, $p(T_f)'g$ vanishes on a dense linear subspace and
therefore $p(T_f)'g=0$. It immediately follows that $T_f'$ has an
eigenvector, which is given by a non-constant function $h\in C(X)$:
$T_f'h=zh$ for some $z\in\C$. Since $T_f$ is invertible, so is
$T_f'$ and therefore $z\neq 0$. It is easy to see that
$T_f'h(n)=h(n+1)$ for each $n\in\Z$. Thus the equality $T_f'h=zh$
implies that (up to a multiplication by a non-zero constant)
$h(n)=z^n$ for each $n\in\Z$. Since $h$ is non-constant, $z\neq 1$.
Thus the map $n\mapsto z^n$ is continuous on $\Z_\tau$ for some
$z\in\C\setminus\{0,1\}$. We have arrive to a contradiction.
\end{proof}

\begin{proof}[Proof of Theorem~$\ref{main}$] Let $\tau$ be
the topology on $\Z$ provided by Lemma~\ref{m2}. By
Proposition~\ref{prop3}, $E=L^0_{\Z_\tau}$ is a complete locally
convex space of countable algebraic dimension. Let $f:\Z\to\Z$,
$f(n)=n+1$ and $S=S_f\in L(E)$. By Lemmas~\ref{m1} and~\ref{m2}, $S$
is invertible and has no non-trivial invariant subspaces. The proof
is complete (with an added bonus of invertibility of $S$).
\end{proof}

\section{Proof of Theorem~\ref{main1}}

We fix the vector space and the operator from the start. The proof
will amount to constructing an appropriate topology on the space.

Let $E\subset \omega=\K^\N$ be the subspace of all sequences with
finite supports. That is, $x\in E$ precisely when
$\supp(x)=\{n\in\N:x_n\neq 0\}$ is finite. In other words,
$E=\spann\{e_n:n\in\N\}$, where $\{e_n\}_{n\in\N}$ is the standard
basis in $\omega$. Clearly $E$ is a dense countably dimensional
subspace of $\omega$. By Theorem~\ref{th1a}, $E$ is strongly
orbital. That is, there is $T\in L(\omega)$ and a hypercyclic vector
$u\in E$ for $T$ such that $E=\spann(O(T,u))$. Since $T\in
L(\omega)$, the matrix of $T$ (with respect to the standard basis)
is row finite. Since $T(E)\subseteq E$, the matrix of $T$ is column
finite. Thus the matrix of $T$ is both row and column finite. It
follows that there is $T^*\in L(\omega)$, whose matrix is the
transpose of the matrix of $T$ and that $T^*(E)\subseteq E$. Of
course, the matrix of $T^*$ is also row and column finite. We shall
use the $\langle\cdot,\cdot\rangle$ notation for the canonical
bilinear form on $\omega\times E$:
$$
\langle x,y\rangle=\sum_{n=1}^\infty x_ny_n.
$$
Clearly, the restriction of this form to $E\times E$ is symmetric
and $\langle Tx,y\rangle=\langle x,T^*y\rangle$ for every $x,y\in
E$. We say that a sequence $\{x_n\}_{n\in\N}$ in $E$ has {\it point
finite supports} if every $k\in\N$ can belong only to finitely many
sets $\supp(x_n)$.

For every sequence $v=\{v_n\}_{n\in\N}$ in $E$ with point finite
supports, we can define the seminorm $p_v$ on $E$:
$$
p_v(x)=\sum_{n=1}^\infty |\langle x,v_n\rangle|.
$$
The point finiteness of the supports of $v_n$ together with
finiteness of the support of $x$ guarantee, that the terms in the
above series are zero for all sufficiently large $n$. By ${\cal P}$,
we shall denote the set of all seminorms on $E$ obtained in this
way. For $p\in{\cal P}$, we denote $p^{[k]}(x)=p(T^kx)$ for
$k\in\Z_+$ and $x\in E$. Clearly, each $p^{[k]}$ is a seminorm on
$E$. Moreover, $p^{[k]}\in {\cal P}$ for every $p\in{\cal P}$ and
$k\in\Z_+$. Indeed, the column finiteness of the matrix of $T^*$
implies that if a sequence $v=\{v_n\}_{n\in\N}$ in $E$ has point
finite supports, then the sequence $\{T^*v_n\}_{n\in\N}$ also has
point finite supports. Thus the seminorm $p_v(T\cdot)$ also belongs
to ${\cal P}$:
$$
p_v(Tx)=\sum_{n=1}^\infty |\langle Tx,v_n\rangle|=\sum_{n=1}^\infty
|\langle x,T^*v_n\rangle|.
$$
We say that a subset $Q$ of ${\cal P}$ is $T$-{\it invariant} if
$p^{[1]}=p(T\,\cdot)$ belongs to $Q$ for every $p\in Q$. Obviously,
if $Q$ is a $T$-invariant subset of ${\cal P}$, then $T:E\to E$ is
continuous when $E$ carries the topology $\tau_Q$ defined by the
collection $Q$ of seminorms. From now on let
$$
\text{$\tau_0$ be the topology on $E$ inherited from $\omega$}.
$$
Obviously, the seminorms $\pi_k(x)=|x_k|$ belong to ${\cal P}$ and
the collection $\{\pi_k:k\in\N\}$ defines (on $E$) the topology
$\tau_0$. Furthermore, $\pi^{[s]}_k$ is also $\tau_0$-continuous for
every $k\in\N$ and $s\in\Z_+$. Let
\begin{equation}\label{p0}
{\cal P}_0=\{\pi^{[s]}_k:k\in\N,\ s\in\Z_+\}.
\end{equation}
Then ${\cal P}_0$ is a countable $T$-invariant subset of ${\cal P}$,
which defines the same topology $\tau_0$.

We say that $y\in\omega$ is $p_v$-{\it singular} if
$$
\sum_{n=1}^\infty |\langle y,v_n\rangle|=\infty.
$$

\begin{lemma}\label{compl} Let ${\cal Q}$ be a subset of ${\cal P}$
such that $\pi_k\in{\cal Q}$ for each $k\in\N$ and for every
$y\in\omega\setminus E$, there is $p\in{\cal Q}$ for which $y$ is
$p$-singular. Then the topological vector space $(E,\tau)$ is
complete, where $\tau=\tau_Q$ is the topology defined by the family
${\cal Q}$ of seminorms.
\end{lemma}

\begin{proof} Clearly $\omega$ is complete and induces on $E$ the
topology weaker than $\tau$. Next, it is easy to see that $\{x\in
E:p(x)\leq 1\}$ is $\tau_0$-closed in $E$ for every $p\in{\cal P}$.
By Lemma~\ref{comple}, the completeness of $(E,\tau)$ will be
verified if we show that for every $\tau$-Cauchy net $\{x_\alpha\}$
in $E$, its $\tau_0$-limit $\lim_{\tau_0} x_\alpha$ belongs to $E$.

Let $\{x_\alpha\}$ be a $\tau$-Cauchy net in $E$ and $x\in\omega$ be
the $\tau_0$-limit of $\{x_\alpha\}$ in $\omega$. We have to show
that $x\in E$. Assume the contrary. Then by our assumptions there is
$p_v\in{\cal Q}$ such that $x$ is $p_v$-singular. Since $p_v\in{\cal
Q}$, $\{p_v(x_\alpha)\}_{\alpha\in D}$ is a Cauchy net in
$[0,\infty)$. Hence there is a non-negative real number $c$ such
that $p_v(x_\alpha)\to c$. Next, since $\{x_\alpha\}$ converges to
$x$ in $\omega$
$$
c=\lim_\alpha p_v(x_\alpha)\geq\lim_\alpha\sum_{n=1}^m |\langle
x_\alpha,v_n\rangle|=\sum_{n=1}^m |\langle x,v_n\rangle|\ \ \
\text{for every $m\in\N$}.
$$
We have arrived to a contradiction with the $p_v$-singularity of
$x$, which proves that $x\in E$.
\end{proof}

Let $G$ be the set of $x\in E$ such that each $x_k$ is rational
($x_k\in \qq+i\qq$ in the case $\K=\C$). Clearly $G$ is a countable
subset of $E$. Furthermore, it is easy to see that
\begin{equation}\label{evto}
\text{$G$ is $\tau$-dense in $E$ for every topology $\tau$ such that
$(E,\tau)$ is a topological vector space}.
\end{equation}
The following lemma is our main instrument.

\begin{lemma}\label{instru} Let $x\in
\omega\setminus E$ and $S$ be a countable subset of ${\cal P}$ such
that $O(T,u)$ is $\tau_S$-dense in $E$ and $\tau_0\subseteq\tau_S$,
where $\tau_S$ is the topology generated by $S$. Then there exists a
sequence $v=\{v_n\}_{n\in\N}$ in $E$ with point finite supports such
that $x$ is $p_v$-singular and $O(T,u)$ is $\tau_{S'}$-dense in $E$,
where $S'=S\cup\{p_v^{[k]}:k\in\Z_+\}$.
\end{lemma}

\begin{proof} Since $S$ is countable, we can write
$S=\{q_k:k\in\N\}$. For each $n\in\N$, let $r_n=q_1+{\dots}+q_n$.
Then $\{r_n\}_{n\in\N}$ is an increasing sequence of seminorms
(belonging to ${\cal P}$) and defining the topology $\tau_S$ on $E$.
Since $G$ is countable, $G=\{g_n:n\in\N\}$. We shall construct
inductively the sequences $\{n_k\}_{k\in\N}$ and $\{m_k\}_{k\in\N}$
of positive integers and $\{v_k\}_{k\in\N}$ in $E$ such that
\begin{align}
&\text{$n_k>n_{k-1}$ and $m_k>m_{k-1}$ if $k\geq 2$;}\label{pno1}
\\
&\text{$x_{m_k}\neq 0$ and $\textstyle
v_k=\frac1{|x_{m_k}|}e_{m_k};$}\label{pno2}
\\
&\textstyle r_k(g_k-T^{n_k}u)<\frac1k;\label{pno3}
\\
&\langle T^l(g_j-T^{n_j}u),v_m\rangle=0\ \ \ \text{for $0\leq
l<j\leq m\leq k$};\label{pno4}
\\
&|\langle T^l(g_j-T^{n_j}u),v_m\rangle|<2^{-j}\ \ \ \text{for $0\leq
l<j\leq k$ and $1\leq m<j\leq k$}.\label{pno5}
\end{align}

{\bf Basis of induction.} Since $O(T,u)$ is $\tau_S$-dense in $E$,
there is $n_1\in\N$ such that $r_1(T^{n_1}u-g_1)<1$. Since $x\notin
E$, we can find $m_1\in\N$ such that $x_{m_1}\neq 0$ and
$m_1\notin\supp(g_1-T^{n_1}u)$. Denote
$v_1=\frac{e_{m_1}}{|x_{m_1}|}$. Since
$m_1\notin\supp(g_1-T^{n_1}u)$, we have $\langle
g_1-T^{n_1}u,v_1\rangle=0$. Clearly (\ref{pno1}--\ref{pno5}) with
$k=1$ are satisfied.

{\bf Induction step.} Assume now that $q\geq 2$ and $n_j,m_j,v_j$
satisfying (\ref{pno1}--\ref{pno5}) for every $k<q$ are already
constructed. We shall construct $n_q$, $m_q$ and $v_q$ such that
(\ref{pno1}--\ref{pno5}) are satisfied with $k=q$. Since $O(T,u)$ is
$\tau_S$-dense in $E$ and $\tau_0\subseteq\tau_S$, we can choose
$n_q\in\N$ such that
\begin{equation}\label{www}
\text{$n_q>n_{q-1}$, $\textstyle r_q(g_q-T^{n_q}u)<\frac1q$ and
$|\langle T^l(g_q-T^{n_q}u),v_m\rangle|<2^{-q}$ \ for $0\leq l<q$
and $1\leq m<q$}.
\end{equation}
Since $x\notin E$, there is $m_q\in\N$ such that
\begin{equation}\label{wwww}
\text{$m_q>m_{q-1}$, $x_{m_q}\neq 0$ and $\textstyle
m_q\notin\bigcup\limits_{0\leq l<j\leq q}\supp(T^l(g_j-T^{n_j}u))$}.
\end{equation}
Since $x_{m_q}\neq 0$, we can define
$v_q=\frac{e_{m_q}}{|x_{m_q}|}$. Clearly, (\ref{pno1}), (\ref{pno2})
and (\ref{pno3}) with $k=q$ are satisfied. The condition
(\ref{pno4}) with $k=q$ follows from (\ref{wwww}) and (\ref{pno4})
with $k<q$. Similarly, (\ref{pno5}) with $k=q$ follows from
(\ref{www}) and (\ref{pno5}) with $k<q$. This concludes the
inductive construction of the sequences $\{n_k\}_{k\in\N}$,
$\{m_k\}_{k\in\N}$ and $v=\{v_k\}_{k\in\N}$ satisfying
(\ref{pno1}--\ref{pno5}).

By (\ref{pno1}) and (\ref{pno2}) the supports of $v_k$ are pairwise
disjoint and therefore are point finite. Hence $p_v\in{\cal P}$. It
remains to show that $O(T,u)$ is $\tau_{S'}$-dense in $E$, where
$S'=S\cup\{p_v^{[k]}:k\in\Z_+\}$. By (\ref{pno3}),
$$
r_k(g_k-T^{n_k}u)\to 0.
$$
Now let $l\in\Z_+$. Using (\ref{pno4}) and (\ref{pno5}), we see that
for $j>l$,
\begin{align*}
p^{[l]}(g_j-T^{n_j}u)&=\sum_{m=1}^\infty |\langle
T^l(g_j-T^{n_j}u,v_m) \rangle|=\sum_{m=1}^j |\langle
T^l(g_j-T^{n_j}u,v_m) \rangle|+\sum_{m=j+1}^\infty |\langle
T^l(g_j-T^{n_j}u,v_m) \rangle|
\\
&=\sum_{m=1}^j |\langle T^l(g_j-T^{n_j}u,v_m) \rangle|<j2^{-j}\to 0.
\end{align*}
According to the last two displays, $g_j-T^{n_j}u\to 0$ in
$\tau_{Q'}$. By (\ref{evto}), $\{g_j:j\in\N\}$ is $\tau_{Q'}$-dense
in $E$. Since $g_j-T^{n_j}u\to 0$ in $\tau_{Q'}$, it follows that
$\{T^{n_j}u:j\in\N\}$ is $\tau_{Q'}$-dense in $E$. Hence $O(T,u)$ is
$\tau_{Q'}$-dense in $E$.
\end{proof}

Now we are ready to prove Theorem~\ref{main1}. Since
$|\omega\setminus E|=2^{\aleph_0}$, the Continuum Hypothesis implies
that $|\omega\setminus E|=\aleph_1$. We shall also denote the first
ordinal of cardinality $\aleph_1$ by the same symbol $\aleph_1$. The
equality $|\omega\setminus E|=\aleph_1$ allows us to enumerate
$\omega\setminus E$ by the ordinals $\alpha$ satisfying $1\leq
\alpha<\aleph_1$: $\omega\setminus E=\{x_\alpha\}_{1\leq
\alpha<\aleph_1}$. We shall construct inductively a chain $\{{\cal
P_\alpha}\}_{\alpha<\aleph_1}$ of subsets of ${\cal P}$ satisfying
the following conditions
\begin{align}
&\text{${\cal P}_0$ is the set defined in (\ref{p0});}\label{tro1}
\\
&\text{${\cal P}_\beta\subseteq{\cal P}_\alpha$ if
$\beta<\alpha$;}\label{tro2}
\\
&\text{${\cal P}_\alpha$ is countable and
$T$-invariant;}\label{tro3}
\\
&\text{there is $p\in {\cal P}_\alpha$ such that $x_\alpha$ is
$p$-singular if $\alpha\geq 1$;}\label{tro4}
\\
&\text{$O(T,u)$ is $\tau_\alpha$-dense in $E$, where $\tau_\alpha$
is the topology defined by ${\cal P}_\alpha$}.\label{tro5}
\end{align}

We will use the following elementary observation.
\begin{equation}\label{buba}
\begin{array}{l}
\text{If $\{\tau_s\}_{s\in S}$ is a totally ordered by inclusion
collection of topologies on the same set $X$,}
\\
\text{then $\tau=\bigcup\tau_s$ is a topology on $X$. Furthermore, a
set $A\subseteq X$ is $\tau$-dense in $X$ if and}
\\
\text{only if $A$ is $\tau_s$-dense in $X$ for each $s\in S$.}
\end{array}
\end{equation}
The set ${\cal P}_0$ serves as the basis of induction. Assume now
that $\gamma<\aleph_1$ and ${\cal P}_\beta$ for $\beta<\gamma$
satisfying (\ref{tro2}--\ref{tro5}) for each $\alpha<\gamma$ are
already constructed. First, set ${\cal
Q}_\gamma=\bigcup\limits_{\beta<\gamma}{\cal P}_\beta$. Then ${\cal
Q}_\gamma$ is $T$-invariant (as a union of $T$-invariant sets) and
countable as a countable union of countable sets. Furthermore, the
topology $\theta$ generated by ${\cal Q}_\gamma$ is the union of
$\tau_\beta$ for $\beta<\gamma$. By (\ref{buba}) and (\ref{tro5})
for $\alpha<\gamma$, $O(T,u)$ is $\theta$-dense in $E$. Since
$x_\gamma\in \omega\setminus E$, all conditions of
Lemma~\ref{instru} with $x=x_\gamma$ and $S={\cal Q}_\gamma$ are
satisfied. Lemma~\ref{instru} provides a sequence
$v=\{v_n\}_{n\in\N}$ in $E$ with point finite supports such that
$x_\gamma$ is $p_v$-singular and $O(T,u)$ is $\tau_\gamma$-dense in
$E$, where $\tau_\gamma$ is generated by ${\cal P}_\gamma={\cal
Q}_\gamma\cup\{p_v^{[k]}:k\in\Z_+\}$. Clearly, the just constructed
set ${\cal P}_\gamma$ satisfies (\ref{tro2}--\ref{tro5}) with
$\alpha=\gamma$. This concludes the construction of the chain
$\{{\cal P_\alpha}\}_{\alpha<\aleph_1}$ of subsets of ${\cal P}$
satisfying (\ref{tro1}--\ref{tro5}). Now we consider
$$
{\cal P}^*=\bigcup_{\alpha<\aleph_1}{\cal P}_\alpha.
$$
Clearly ${\cal P}^*$ is $T$-invariant as a union of $T$-invariant
sets. Hence $T$ is a continuous linear operator on $(E,\tau^*)$,
where the topology $\tau^*$ defined by ${\cal P}^*$. Furthermore,
since ${\cal P}_\alpha$ are totally ordered by inclusion, $\tau^*$
coincides with $\bigcup \tau_\alpha$ for $\alpha<\aleph_1$.
According to (\ref{tro5}) and (\ref{buba}), $O(T,u)$ is
$\tau^*$-dense in $E$. Hence $u$ is a hypercyclic vector for $T\in
L(E_{\tau^*})$. Finally, by (\ref{tro2}) and (\ref{tro4}), for every
$x\in\omega\setminus E$, there is $p\in{\cal P}^*$ such that $x$ is
$p$-singular. Now Lemma~\ref{compl} guarantees the completeness of
$(E,\tau^*)$. Thus $E_{\tau^*}=(E,\tau^*)$ is a complete locally
convex space, $T\in L(E_{\tau^*})$, $u$ is a hypercyclic vector for
$T$ and $E=\spann(O(T,u))$. Thus $E_{\tau^*}$ is complete and
strongly orbital.

\section{Open problems}

We have characterized orbital and strongly orbital metrizable
locally convex spaces. The following more general problems remains
open.

\begin{question}\label{q1} Characterize orbital countably
dimensional locally convex spaces. Characterize orbital countably
dimensional complete locally convex spaces.
\end{question}

\begin{question}\label{q2} Characterize strongly orbital
countably dimensional locally convex spaces. Characterize strongly
orbital countably dimensional complete locally convex spaces.
\end{question}

\begin{question}\label{q3} Characterize
countably dimensional locally convex spaces having the invariant
subspace/subset property. Characterize complete countably
dimensional locally convex topological vector spaces having the
invariant subspace/subset property.
\end{question}

\begin{question}\label{q4} Can one get rid of the
Continuum Hypothesis in Theorem~\ref{main1}\,{\rm ?}
\end{question}

%

\vfill\break

\small\rm

\vskip1truecm

\scshape

\noindent Stanislav Shkarin

\noindent Queens's University Belfast

\noindent Pure Mathematics Research Centre

\noindent University road, Belfast, BT7 1NN, UK

\noindent E-mail address: \qquad {\tt s.shkarin@qub.ac.uk}

\end{document}